\documentclass[10pt,a4paper]{article}

\usepackage{amsmath,amssymb,amsfonts,theorem}
\usepackage{graphicx,epsfig}
\usepackage{subfigure,psfrag}
\usepackage[ruled]{algorithm}
\usepackage{algorithmic}
\usepackage{empheq}
\usepackage[normalem]{ulem}
\usepackage[dvipsnames]{xcolor}
 
\graphicspath{{./Figs/}}

\def\1{\mathbf{1}}

\newcommand\oprocendsymbol{\hbox{$\square$}}
\newcommand\oprocend{\relax\ifmmode\else\unskip\hfill\fi\oprocendsymbol}

\def\K{\mathcal{K}}
\def\KL{\mathcal{KL}}

\def\N{\mathbb{N}}
\def\R{\mathbb{R}}
\def\Tset{\mathbb{T}}

\newcommand{\T} {^\top}

\newcommand{\dist}{\operatorname{dist}}             

\newtheorem{theorem}{Theorem}[section]
\newtheorem{proposition}[theorem]{Proposition}
\newtheorem{corollary}[theorem]{Corollary}
\newtheorem{definition}[theorem]{Definition}
\newtheorem{lemma}[theorem]{Lemma}

{\theorembodyfont{\rmfamily} 

}

\newenvironment{proof}{\sl{Proof: }\color{Black!40!black}}{\hfill$\square$\\}

\begin{document}
\begin{center} {Universit\`a  degli studi di Padova\\ Dipartimento di Ingegneria dell'Informazione}\end{center}
\begin{center} Nicoletta Bof, Ruggero Carli, Luca Schenato\end{center}
\begin{center}{\bf Technical Report (Sept. 2018)}\\ 
\vspace{0.5cm}

{\bf \Large  Lyapunov Theory for Discrete Time Systems}
\end{center}
\vspace{0.5cm}

This work contains a collection of Lyapunov related theorems for discrete time systems. Its main purpose it to collect in a self contained document part of the Lyapunov theory in discrete time, since, in the literature, there does not seem to be a unique work which contains these results and their proof, apart from \cite{Hahn}, which deals with discrete time Lyapunov theory, but is written in German, and so it is not easily accessible. The work has been obtained starting from the Lyapunov results for continuous time given  in \cite{Khalil} and from the results contained in \cite{Hahn}. 

The work also contains some convergence results for particular class of discrete time systems (in Sections 6 and 7).

Before focusing on Lyapunov theory, it is useful to introduce the following definition

\begin{definition}
A function $f(t,x)$ is said to be \textbf{Lipschitz} in $(\bar{t},\bar{x})$ if 
\begin{equation} \label{eq:Lips}
\| f(t,x)-f(t,y)\| \leq L \|x-y\| 
\end{equation}
$\forall\ (t,x),(t,y)$ in a neighbourhood of $(\bar{t},\bar{x})$.
The constant $L$ is called \textbf{Lipschitz constant}.\\
Consider $f(t,x)=f(x)$ independent of $t$
\begin{itemize}
 \item  $f$ is \textbf{locally Lipschitz} on a domain $D\subset \R^n$ open and connected if each point in $D$ has a neighbourhood $D_0$ such that \eqref{eq:Lips} is satisfied with Lipschitz constant $L_0$.
 \item $f$ is \textbf{Lipschitz} on a set $W$ if \eqref{eq:Lips} is satisfied for all points in $W$ with the same constant $L$. A function $f$ locally Lipschitz on $D$ is 	Lipschitz on every compact subset of $D$.
 \item $f$ is \textbf{globally Lipschitz} if it is Lipschitz on $\R^n$.
 \end{itemize} 
If $f(t,x)$ depends on $t$, the same definitions are said to hold \textbf{uniformly in $t$} for all $t$ in an interval of time, if  the Lipschitz constant do not vary due to the time. In particular, a function $f(t,x)$ is \textbf{globally uniformly Lipschitz} if  it is Lipschitz on $\R^n$ and the constant $L$ does not depend  on time $t$.
\end{definition}
A continuously differentiable function on a domain $D$ is also Lipschitz in the same domain.
\section{Autonomous systems}
Consider the autonomous system
\begin{equation} \label{eq:autonomousSyst}
x(t+1) = f(x(t))
\end{equation}
where $f\colon D\to\R^n$ is {\bf locally Lipschitz} in $D\subset \R^n$, and suppose $f(0)=0$, that is $x=0$ is an equilibrium point for system \eqref{eq:autonomousSyst} (all this can be extended for an equilibrium point different from $0$).
\begin{definition}
The equilibrium point $x=0$ of \eqref{eq:autonomousSyst} is
\begin{itemize}
\item \textbf{stable} if, for each $\epsilon>0$, there is $\delta=\delta(\epsilon)$ such that 
\[
	\|x(0)\|<\delta\Rightarrow\|x(t)\|<\epsilon,\ \forall t\geq 0
\]
\item \textbf{unstable} if it is not stable
\item \textbf{asymptotically stable} if it is stable and $\delta$ can be chosen such that
\[
	\|x(0)\|<\delta\Rightarrow \lim_{t\to\infty}x(t) = 0
\]
\end{itemize}
\end{definition}

\begin{theorem}[Existence of a Lyapunov function implies stability]\label{th:LyapAutSyst}
Let $x=0$ be an equilibrium point for the autonomous system
\[
	x(t+1) = f(x(t))
\]
where $f\colon D\to\R^n$ is locally Lipschitz in $D\subset \R^n$ and $0\in D$. Suppose there exists a function $V\colon D\to\R$ which is continuous and such that
\begin{gather} V(0) = 0 \text{ and } V(x)>0,\ \forall x\in D-\{0\} \label{eq:VposDef}\\
V(f(x)) -V(x) \leq 0,\ \forall x\in D \label{eq:Vdecr}
\end{gather}
Then $x=0$ is stable. Moreover if 
\begin{equation}
V(f(x)) -V(x) < 0,\ \forall x\in D-\{0\} \label{eq:Vstrdecr}
\end{equation}
then $x= 0$ is asymptotically stable.
\end{theorem}
\begin{proof}
Given $\epsilon>0$, choose $r\in(0,\epsilon]$ such that $B_r=\{x\in\R^n\ |\ \|x\|<r\}\subset D$. Let $\alpha = \min_{\|x\|=r}V(x)$, then $\alpha>0$ by \eqref{eq:VposDef}. Take $\beta \in (0,\alpha)$ and let $\Omega_\beta = \{x\in B_r\ |\ V(x)\leq\beta\}$. $\Omega_\beta$ is in the interior of $B_r$, and any trajectory that starts in $\Omega_\beta$ stays in $\Omega_\beta$ for all $t\leq 0$. This is true due to \eqref{eq:Vdecr}, since $V(x(t+1))\leq V(x(t))\leq \dots\leq V(x(0))\leq \beta,\ \forall t\geq 0$. Since $\Omega_\beta$ is compact and invariant for $f$, and $f$ is locally Lipschitz in $D$, there exists a unique solution defined for all $t>0$ if $x(0)\in\Omega_\beta$. Since $V(x)$ is continuous and $V(0)=0$, there exists a $\delta>0$ such that if $\|x\|<\delta$ then $V(x)<\beta$. This implies that $B_\delta\subset \Omega_\beta \subset B_r$, so
\[x(0)\in B_\delta \Rightarrow x(0) \in \Omega_\beta \Rightarrow x(t)\in\Omega_\beta,\ \forall t\geq 0 \Rightarrow x(t)\in B_r,\ \forall t\geq 0\]
Therefore
\[
	\|x\|<\delta \Rightarrow \|x(t)\|<r\leq \epsilon,\ \forall t\geq 0
\]
and so $x=0$ is a stable point.\\
Suppose now that \eqref{eq:Vstrdecr} holds. Since $x=0$ is stable, it is possible to find for every $r>0$ such that $B_r\subset D$ a constant $b$ such that $\Omega_b\subset B_r$ and all the trajectorys starting from $\Sigma_b$ stay in $\Sigma_b$.  Since $V(x)$ is bounded below by $0$ and is strictly decreasing along the trajectories in $D$, it holds that $V(x(t))\to c\geq 0$ as $t\to\infty$. Suppose ab absurdo that $c>0$. Starting from a point in $\Omega_b$ we have that the trajectories are such that $V(x(t))\to c$ as $t\to\infty$ with $c<b$ (otherwise if $c=b$ starting from a point $x\in\Omega_b$ such that $V(x)=b$ the trajectory would reach in a step a point such that $V(f(x))<b$, which is a contradiction). Now, as before, since $V(x)$ is continuous and $V(0)=0$, there exists $d>0$ such that $B_d\subset\Omega_c$. Consider set $\Delta = \{x|d\leq\|x\|\leq r\}$, which is a compact set that contains $\{x|V(x)=c\}$. Since $V$ and $f$ are continuous functions, we can define
\[\gamma := \min_{x\in\Delta} V(x)-V(f(x))\]
due to Bolzano Weierstrass theorem. Since $\lim_{t\to\infty}V(x(t))= c$ and $V$ is continuous, there exists $\bar{t}$ such that for all $t>\bar{t},\ V(x(t))\leq c+\gamma'$, with $\gamma'<\gamma$ and $x(t)\in\Delta$. Since $x(t)$ belongs to $\Delta$, it also holds that $V(x(t))-V(f(x(t)))\geq\gamma$, and so $V(f(x(t)))\leq -\gamma+V(x(t))\leq c+\gamma'-\gamma<c$, which is a contradiction.
\end{proof}
\begin{definition} A function $V\colon D\to\R$ satisfying \eqref{eq:VposDef} and \eqref{eq:Vdecr} is called  a \textbf{Lyapunov function}.
\end{definition}
\begin{theorem}[Global asymptotic stability from Lyapunov]\label{the:globallyStable}
Let $x=0$ be an equilibrium point for the autonomous system
\[
	x(t+1) = f(x(t))
\]
where $f\colon D\to\R^n$ is locally Lipschitz in $D\subset \R^n$ and $0\in D$. Let $V\colon \R^n\to\R$ be a continuous function such that
\begin{align}
V(0) = 0 \text{ and } V(x)>0,\ \forall x\in D-\{0\} \label{eq:VposDef1}\\
\|x\|\to\infty \Rightarrow V(x)\to\infty \label{eq:VradUnbounded}\\
V(f(x)) -V(x) < 0,\ \forall x\in D \label{eq:Vdecr1}
\end{align}
then $x=0$ is globally asymptotically stable.
\end{theorem}
\begin{proof}
Given any point $p\in\R^n$, let $c=V(p)$. Due to \eqref{eq:VradUnbounded}, for any $c>0$ there is $r>0$ such that $V(x)>c$ whenever $\|x\|>r$. Thus $\Omega_c\subset B_r$, and we can proceed as done in Theorem \ref{th:LyapAutSyst}
\end{proof}

If an equilibrium point is globally asymptotically stable, than it is the only possible equilibrium point of the system \eqref{eq:autonomousSyst}.\\
The following gives a condition for instability.
\begin{theorem}[Instability condition from Lyapunov]\label{theo:instability}
Let $x=0$ be an equilibrium point for the autonomous system
\[
	x(t+1) = f(x(t))
\]
where $f\colon D\to\R^n$ is locally Lipschitz in $D\subset \R^n$ and $0\in D$. Let $V\colon D\to\R$ be a continuous function such that $V(0)=0$ and $V(x_0)>0$ for some $x_0$ with arbitrary small $\|x_0\|$. Let $r>0$ be such that $B_r\subset D$ and $U=\{x\in B_r|V(x)>0\}$, and suppose that $V(f(x))-V(x)>0$ for all $x\in U$. Then $x=0$ is unstable.
\end{theorem}
\begin{proof} Consider $x_0\in U$ and let $a=V(x_0)>0$. The set $\Sigma_a=\{x\in U| V(x)\geq a\}$ is compact, so we can define $\alpha = \min_{x\in\Sigma} (V(f(x))-V(x))$. There exists an instant $\bar{t}$ such that $x(t)\in U$ for $0\leq t<\bar{t}$ and $\|x(\bar{t})\|>r$ for $t=\bar{t}$. This holds because $V(f(x))>V(x)+\alpha,\ \forall x\in U$ so, at time instant $t$, $V(x(t+1)=V(f(x(t)))>V(x(t))>0$. Now if $\|x(t+1)\|\leq r$, then the trajectory is still in $U$, otherwise is in $D\setminus B_r$. The latter is true because to be in $B_r\setminus U$, $V(x(t+1))$ has to be smaller than 0, but $V(x(t+1))>0$. So starting from a point arbitrarily close to the origin the trajectory goes outside $B_r$, and therefore the origin is unstable.
\end{proof}
\section{The invariance principle}
\begin{definition}
A point $p$ is said to be a \textbf{positive limit point} of $x(t)$ if there is a sequence $\{t_n\}$, with $t_n\to\infty$ as $n\to\infty$ such that $x(t_n)\to p$ as $n\to\infty$. The set of all positive limit points of $x(t)$ is called \textbf{positive limit set} of $x(t)$. 
\end{definition}
\begin{definition}
A set $M$ is an \textbf{invariant set} with respect to \eqref{eq:autonomousSyst} if $x(0)\in M \Rightarrow x(t)\in M,\ \forall\ t\in \R$. It is a \textbf{positive invariant set} if $x(0)\in M \Rightarrow x(t)\in M,\ \forall\ t\geq 0$. 
\end{definition}
The trajectory $x(t)$ approaches $M$ as $t\to\infty$, if for each $\epsilon>0$ there is $T>0$ such that $\dist(x(t),M)<\epsilon,\ \forall\ t>T$, where $\dist(p,M) = \inf_{x\in M}\|p-x\|$. Note that this does not imply that $\lim_{t\to\infty}$ exists.\\
Equilibrium points and limit cycles are example of invariant sets for \eqref{eq:autonomousSyst}, and if a Lyapunov function $V$ for the latter system exists, then also the set $\Omega_c=\{x\in D|V(x)\leq c\}$ is an invariant set.

\begin{lemma} If a solution $x(t)$ of \eqref{eq:autonomousSyst} is bounded and belongs to $D$ for $t\geq 0$, then its positive limit set $L^+$ is a nonempty, compact, invariant set. Moreover, $x(t)$ approaches $L^+$ as $t\rightarrow\infty$.
\end{lemma}
\begin{proof}
The proof can be obtained from appendix C3 of \cite{Khalil} 
\end{proof}

The following is known as LaSalle's theorem
\begin{theorem}[LaSalle's theorem]
Let $\Omega\subset D$ be a compact set that is positively invariant with respect to the autonomous system
\[
	x(t+1) = f(x(t))
\]
where $f\colon D\to\R^n$ is locally Lipschitz in $D\subset \R^n$ and $0\in D$. Let $V:D\rightarrow\R$ be a continuous function such that $V(f(x))-V(x)\leq 0$ in $\Omega$. Let $E$ be the set of all points in $\Omega$ where $V(f(x))-V(x)=0$, and let $M$ be the largest invariant set in $E$. Then every solution starting in $\Omega$ approaches $M$ as $t\to\infty$.
\end{theorem}
\begin{proof}
Can be obtained from proof of Theorem 4.4 page 128 in \cite{Khalil}.
\end{proof}

\begin{corollary}
Let $x=0$ be an equilibrium point for the autonomous system
\[
	x(t+1) = f(x(t))
\]
where $f\colon D\to\R^n$ is locally Lipschitz in $D\subset \R^n$ and $0\in D$. Let $V\colon D\rightarrow\R$ be a continuous positive definite function on a domain $D,\ x\in D$, such that $V(f(x))-V(x)\leq 0$ in $D$. Let $S=\{x\in D|V(f(x))-V(x)= 0\}$ and suppose that no solution can stay identically in $S$ other than the trivial solution $x(t)\equiv 0$. Then the origin is asymptotically stable.
\end{corollary}

\begin{corollary}
Let $x=0$ be an equilibrium point for the autonomous system
\[
	x(t+1) = f(x(t))
\]
where $f\colon D\to\R^n$ is locally Lipschitz in $D\subset \R^n$ and $0\in D$. Let $V\colon \R^n\rightarrow\R$ be a continuous, positive definite, radially unbounded function, such that $V(f(x))-V(x)\leq 0,\ \forall x\in\R^n$. Let $S=\{x\in \R^n|V(f(x))-V(x)= 0\}$ and suppose that no solution can stay identically in $S$ other than the trivial solution $x(t)\equiv 0$. Then the origin is {\bf globally}  asymptotically stable.
\end{corollary}

LaSalle principle is useful because
\begin{itemize}
\item it gives an estimate of the region of attraction of the equilibrium point. It can be any compact positively invariant set;
\item there is an equilibrium set and not an isolated equilibrium point;
\item function $V(x)$ does not have to be positive definite;
\item in case of the corollaries it relaxes the negative definiteness on $V(f(x))-V(x)$.
\end{itemize}

Before going to the linearisation part, we prove the following theorem concerning systems with exponentially asymptotic equilibrium points (see \cite[Ex. 4.68]{Khalil})
\begin{theorem}[Exponential stability  implies existence of a Lyapunov function]\label{theo:expStabImpliesPresenceofLyapAuto}
Let $x=0$ be an equlibrium point for the nonlinear system the autonomous system
\[
	x(t+1) = f(x(t))
\]
where $f\colon D\to\R^n$ is continuously differentiable and $D=\{x\in\R^n|\ \|x\|<r\}$. Let $k,\lambda,$ and $r_0$ be positive constants with $r_0<r/k$. Let $D_0=\{x\in\R^n|\ \|x\|<r_0\}$. Assume that the solutions of the system satisfy
\begin{equation}\label{eq:exponentialStability}
	\|x(t)\|\leq k\|x(0)\|e^{-\lambda t},\ \forall x(0)\in D_0,\ \forall t\geq 0
\end{equation}
Show that there is a function $V\colon D_0\to\R$ that satisfies
\[
	\begin{array}{c}
	c_1\|x\|^2\leq V(x) \leq c_2 \|x\|^2\\
	V(f(x))-V(x)\leq -c_3 \|x\|^2\\
	|V(x)-V(y)|\leq c_4 \|x-y\|(\|x\|+\|y\|)
	\end{array}
\]
for all $x,y\in D_0$ and for some positive constants $c_1,c_2,c_3$ and $c_4$.
\end{theorem}
\begin{proof}
Let $\phi(t,x)$ be the solution of $x(t+1)=f(x(k))$ at time $t$ starting from $x(0)=x$ at time $k=0$. Let
\[
	V(x) = \sum_{t=0}^{N-1}\phi\T(t,x)\phi(t,x)
\]
for some integer variable $N$ to be set. Then 
\[V(x) = x\T x+\sum_{t=1}^{N-1}\phi\T(t,x)\phi(t,x)\geq x\T x = \|x\|\]
and on the other hand, using \eqref{eq:exponentialStability} we have
\[
	V(x) = \sum_{t=0}^{N-1}x(t)\T x(t)\leq  \sum_{t=0}^{N-1} k^2\|x\|^2e^{-2\lambda t}\leq k^2\left(\frac{1-e^{-2\lambda N}}{1-e^{-2\lambda}}\right)\|x\|^2
\]
We have shown that there exists $c_1$ and $c_2$ such that
\[
	c_1\|x\|^2\leq V(x) \leq c_2 \|x\|^2
\]
is satisfied. Now, since $\phi(t,f(x)) = \phi(t,\phi(1,x))=\phi(t+1,x)$,
\begin{align*}
	V(f(x))-V(x) &= \sum_{t=0}^{N-1}\phi\T(t+1,x)\phi(t+1,x)-\sum_{t=0}^{N-1}\phi\T(t,x)\phi(t,x)=\\
	 			&=\sum_{j=1}^{N}\phi\T(j,x)\phi(j,x)-\sum_{t=0}^{N-1}\phi\T(t,x)\phi(t,x) = \phi\T(N,x)\phi(N,x)-x\T x\\
	 			&\leq k^2e^{-2\lambda N}\|x\|^2-\|x\|^2 = -(1-k^2e^{-2\lambda N})\|x\|^2
\end{align*}
Now we can choose $N$ big enough so that $1-k^2e^{-2\lambda N}$ is greater than $0$ and also the second property has been proven. For the third property, since $f$ is continuously differentiable it is also Lipschitz over the bounded domain $D$, with a Lipschitz constant $L$, for which it holds $\|f(x)-f(y)\|\leq L\|x-y\|$. Then 
\[
	\|\phi(t+1,x)-\phi(t+1,y)\|=\|f(\phi(t,x))-f(\phi(t,y))\|\leq L\|\phi(t,x)-\phi(t,y)\|
\]
and by induction
\[
	\|\phi(t,x)-\phi(t,y)\|\leq L^t\|x-y\|
\]
Consider now
\begin{align*}
	|V(x)-V(y)|&=\left\vert \sum_{t=0}^{N-1}(\phi\T(t,x)\phi(t,x)-\phi\T(t,y)\phi(t,y))\right\vert\\
			& = \left\vert \sum_{t=0}^{N-1}[\phi\T(t,x)(\phi(t,x)-\phi(t,y))+\phi\T(t,y)(\phi(t,x)-\phi(t,y))]\right\vert\\
			&\leq \sum_{t=0}^{N-1}[\|\phi\T(t,x)\|\|\phi(t,x)-\phi(t,y)\|+\|\phi\T(t,y)\|\|\phi(t,x)-\phi(t,y)\|]\\
			&\leq \sum_{t=0}^{N-1}[\|\phi\T(t,x)\|+\|\phi\T(t,y)\|]L^t\|x-y\|\\
			&\leq \left[\sum_{t=0}^{N-1} ke^{-\lambda t} L^t\right](\|x\|+\|y\|)\|x-y\|\\
			&\leq c_4 (\|x\|+\|y\|)\|x-y\|
\end{align*}
and so we have proven the last inequality.
\end{proof}

\section{Linear systems and Linearization}
Consider the linear time-invariant system
\begin{equation} \label{eq:linSystTimeInvariant}
x(t+1) = Ax(t),\ A\in\R^{n\times n}
\end{equation}
It has an equilibrium point in the origin $x=0$. The solution of the linear system starting from $x_0\in\R^n$ has the form
\[
	x(t) = A^tx(0)
\]
We have the following result on the stability of linear systems
\begin{theorem} \label{th:stabilityEigval}
The equilibrium point $x=0$ of the linear time-invariant system
\[
x(t+1) = Ax(t),\ A\in\R^{n\times n}
\]
is stable if and only if all the eigenvalues of $A$ satisfy $|\lambda_i|\leq 1$ and the algebraic and geometric multiplicity of the eigenvalues with absolute value 1 coincide. The equilibrium point $x=0$ is globally asymptotically stable if and only if all the eigenvalues of $A$ are such that $|\lambda_i|<1$.
\end{theorem} 

A matrix $A$ with all the eigenvalues in absolute value smaller than 1 is called a \textbf{Schur matrix}, and it holds that the origin is asymptotically stable if and only if matrix $A$ is Schur.

To use Lyapunov theory for linear system we can introduce the following candidate
\[
	V(x) = x\T Px
\]
with $P$ a symmetric positive definite matrix. It's total difference is
\[
	V(f(x))-V(x) = x\T A\T PAx-x\T Px = x\T(A\T PA-P)x:=-x\T Qx
\]
Using Theorem \ref{th:LyapAutSyst} we have that if $Q$ is positive-semidefinite the origin is stable, whether if $Q$ is positive definite the origin is asymptotically stable.
Fixing a positive definite matrix $Q$, if the solution of the Lyapunov equation
\begin{equation}\label{eq:lyapEq}
	A\T PA-P = -Q
\end{equation}
with respect to $P$ is positive definite, then the trajectories converge to the origin.
\begin{theorem}[Lyapunov for linear time invariant systems]
 A matrix $A$ is Schur if and only if, for any positive definite matrix $Q$ there exists a positive definite symmetric matrix $P$ that satisfies \eqref{eq:lyapEq}. Moreover if $A$ is Schur, then $P$ is the unique solution of \eqref{eq:lyapEq}.
\end{theorem}
\begin{proof} Sufficiency can be obtained combining Theorem \ref{th:stabilityEigval} and the fact that the existence of solution $P$ for any positive definite matrix $Q$ assures the convergence of the trajectory. Suppose now that $A$ is Schur stable and build matrix $P$ as
\begin{equation}\label{eq:Pconstr}
	P = \sum_{t=0}^{\infty} (A\T)^t Q A^t
\end{equation}
for any positive definite $Q$. Matrix $P$ is symmetric and positive definite since $Q$ is positive definite. We need to show that using this $P$ equation \eqref{eq:lyapEq} is satisfied. Substituting  \eqref{eq:Pconstr} in \eqref{eq:lyapEq} we obtain
\[
	A\T \sum_{t=0}^{\infty} \left[(A\T)^t Q A^t\right] A-\sum_{t=0}^{\infty} (A\T)^t Q A^t = \sum_{t=0}^{\infty} \left[(A\T)^{t+1} Q A^{t+1}-(A\T)^t Q A^t\right] = -Q
\]
Suppose now that $P$ is not unique, so there exists a $\tilde{P}\neq P$ such that \eqref{eq:lyapEq} is satisfied. So it holds
\[
	A\T PA-P = A\T \tilde{P}A-\tilde{P} \Leftrightarrow A\T (P-\tilde{P}) A-(P-\tilde{P}) = 0
\]
from which, defining $R(x):= x\T(P-\tilde{P})x$ it follows that 
\[
	R(f(x))-R(x) = 0,\ \forall x\in\R^n \Leftrightarrow R(x(0)) = R(x(t)),\ \forall t\geq 0.
\]
Now it holds that
\[
	\lim_{t\to\infty} R(x(t)) = \lim _{t\to\infty} x(0)\T (A\T)^t (P-\tilde{P})A^t x(0)=0,\ \forall x\in\R^n
\]
since $A$ is Schur stable, so due to the fact that $R(x(0)) = R(x(t))$, we have that 
\[
	x\T(P-\tilde{P})x = 0,\ \forall x\in\R^n \Leftrightarrow P-\tilde{P} = 0
\]
and so the solution is unique.
\end{proof}

Let us consider again the nonlinear model
\[
	x(t+1) = f(x(t))
\]
with $f\colon D\to\R^n$ a continuously differentiable map from $D\subset\R^n,\ 0\in D$ into $\R^n$ such that $f(0)=0$. Using the mean value theorem, each component of $f$ can be rewritten in the following form
\[
	f_i(x) = \frac{\partial f_i}{\partial x}(z_i) x 
\]
for some $z_i$ on the segment from the origin to $x$. It is valid for any $x\in D$, where the line connecting $x$ to the origin entirely belongs to $D$. We can also write
\[
	f_i(x) = \frac{\partial f_i}{\partial x}(0) x+\underbrace{\left[ \frac{\partial f_i}{\partial x}(z_i)  -\frac{\partial f_i}{\partial x }(0) \right]x}_{g_i(x)}
\] 
where each $g_i(x)$ satisfies
\[
	|g_i(x)|\leq \left\Vert\frac{\partial f_i}{\partial x}(z_i) -\frac{\partial f_i}{\partial x}(0)\right\Vert\|x\| 
\]
Function $f$ can be rewritten as 
\[
	f(x)= Ax+g(x)
\]
where $A=\frac{\partial f_i}{\partial x}(0)$. By continuity of $\frac{\partial f_i}{\partial x}$ we have that
\[
	\frac{\|g(x)\|}{\|x\|}\to 0 \text{ as } \|x\|\to 0
\] 
Therefore, in a small neighbourhood of the origin the nonlinear system can be approximated by $x(t+1) = Ax(t)$.

The following is known as Lyapunov's indirect method.
\begin{theorem}[Linearised asympt stable implies nonlin asympt stable] Let $x=0$ be an equilibrium point for the nonlinear autonomous system
\[
	x(t+1) = f(x(t))
\]
where $f\colon D\to\R^n$ is locally Lipschitz in $D\subset \R^n$ and $0\in D$. Let $A=\left.\frac{\partial f_i}{\partial x}(x)\right|_{x=0}$. Then the origin is asymptotically stable if $|\lambda_i|<0$ for all the eigenvalues of $A$. Instead, if there exists at least an eigenvalue such that $|\lambda_i|>0$, then the origin is unstable.
\end{theorem}
\begin{proof}
Since $A$ is stable there exists a positive definite matrix $P,\ p_1 I\leq P\leq p_2 I$, such that $V(x)$ is a Lyapunov function for the linearised system $x(t+1) = Ax(t)$, and so it solves \eqref{eq:lyapEq} for any positive definite matrix $Q,\ q_1 I\leq Q\leq q_2 I$. Applying the same Lyapunov function to the nonlinear system we get the following total difference
\begin{align*}
	V(f(x))-V(x)&=f(x)\T Pf(x)-x\T P x=(Ax+g(x))\T P(Ax+g(x))-x\T Px \\
	&= x\T A\T P Ax-x\T Px+2g(x)\T Px+g(x)\T P g(x)\\ 
	&=-x\T Q x+2g(x)\T Px+g(x)\T P g(x)\\
\end{align*}
Since $\frac{\|g(x)\|}{\|x\|}\to 0$ as $\|x\|\to 0$, fixed a constant $\gamma>0$, there exists a neighbourhood of $x,\ \|x\|<\epsilon$ such that $\|g(x)\|<\gamma \|x\|$, and so
\[
	V(f(x))-V(x)\leq -q_1\|x\|^2+p_2\gamma^2 \|x\|^2+2p_2\gamma\|x\|^2=(-q_1+p_2\gamma^2+2p_2\gamma)\|x\|^2,\ \forall \|x\|<\epsilon
\]
Therefore, choosing $\gamma$ such that $-q_1+p_2\gamma^2+2p_2\gamma$ is negative, $V(x)$ is indeed a Lyapunov function for the starting nonlinear system. Note that $-q_1+p_2\gamma^2+2p_2\gamma=0$ describes in $\gamma$ a parabola whose vertex is in the third quarter and is directed towards the upper part of the plane, so there exists a $\gamma$ which satisfy the property required.\\

To show the instability part, we first give the following statement regarding the solvability of the discrete Lyapunov equation
\begin{lemma}B
The Lyapunov equation \eqref{eq:lyapEq} admits a solution if and only if the eigenvalues $\lambda_i$ of matrix $A$ are such that 
\begin{equation}\label{eq:conditionSolvabilityLyap}
	\lambda_i\lambda_j\neq 1 \text{ for all } i,j=1,\dots,n
\end{equation}
Moreover given a positive definite matrix $Q$, the corresponding solution $P$ is positive definite if and only if $|\lambda_i|<1$ for all $i=1,\dots,n$. 
\end{lemma}
The proof of the previous lemma can be found in \cite{Hahn}. Suppose now that there is $\lambda_i$ such that $|\lambda_i|>1$ but that condition \eqref{eq:conditionSolvabilityLyap} is satisfied. Therefore, given a positive definite matrix $Q$, the corresponding solution $P$ of $A\T P A-P = -Q$ is not positive semi-definite (note that if $Q$ is positive definite and $A$ is Schur stable, then $P$ cannot be positive semi-definite, since if $x$ is such that $x\T P x=0$, then $x\T A\T PAx<0$ but this is not possible since $P$ is positive semi-definite). Matrix $\tilde{P}=-P$ is not negative semi-definite, so defining $V(x) = x\T\tilde{P}x$, it holds $V(x)>0$ for some $x\in\R^n$; for the same vector $x$ it also holds 
\begin{align*}
	V(f(x))-V(x)&=x\T A\T \tilde{P}Ax-x\T \tilde{P} x = x\T (A\T \tilde{P}A-\tilde{P})x = \\
		&-x\T (A\T P A-P)x = -x\T (-Q)x= x\T Qx>0
\end{align*}
Now we can apply Theorem \ref{theo:instability} and conclude that the equilibrium point is instable. If condition 	\eqref{eq:conditionSolvabilityLyap} is not satisfied, consider matrix $A_1=\frac{1}{\gamma}A$, with $\gamma$ such that \eqref{eq:conditionSolvabilityLyap} is satisfied using matrix $A_1$ and the solution matrix $P_1$ for \eqref{eq:lyapEq} for any positive definite matrix $Q$ is not positive semi-definite (that is there is at least one eigenvalue of $A_1$ with modulus greater than $1$). Therefore, choosing $V(x) = -x\T P_1 x$, it holds that $V(x)>0$ for some $x$. It also holds that 
\begin{align*}
V(f(x))-V(x) &= -x\T A\T P_1 Ax+x\T P_1x = -x\T (\gamma^2 A_1\T P_1 A_1- P_1-\gamma^2 P_1+\gamma^2 P_1) x & \\
&= -\gamma^2 x\T(A_1\T P_1 A_1-P_1 )x - (\gamma^2-1) x\T P_1 x \\
&= -\gamma^2 x\T (-Q)x- (\gamma^2-1) x\T P_1 x = \gamma^2 x\T Qx- (\gamma^2-1) x\T P_1 x
\end{align*}
Choosing an adequately small $\gamma$ it holds $V(f(X))-V(x)>0$ and we can apply again Theorem \ref{theo:instability}.
\end{proof}
 This theorem allows to find a Lyapunov function for the nonlinear system in a neighbourhood of the origin, provided that the linearised system is asymptotically stable.
 \section{Comparison Functions}
Consider the nonautonomous system
\[
	x(t+1) = f(t,x(t))
\]
starting from $x(t_0)=x_0$ at time $t_0$, with $f\colon (\Tset\times D)\to\R^n,\ \Tset=\{t_0,t_0+1,\dots\},\ D\subset\R^n$. The evolution of the system depends on the starting time $t_0$.  We need new definitions for the stability in order to have them hold uniformly in the initial time $t_0$. We will exploit the following class of functions
\begin{definition}A continuous function $\alpha\colon [0,a)\to[0,\infty)$ is said to belong to class $\K$ if it is strictly increasing and $\alpha(0)=0$. It is said to belong to class $\K_{\infty}$ if $a=\infty$ and $\alpha(r)\to\infty$ as $r\to\infty$.
\end{definition}
\begin{definition} A continuous function $\beta\colon [0,a)\times [0,\infty)\to[0,\infty)$ is said to belong to class $\KL$ if, for each fixed $s$, the mapping $\beta(r,s)$ belongs to class $\K$ with respect to $r$ and, for each fixed $r$, the mapping $\beta(r,s)$ is decreasing with respect to $s$ and $\beta(r,s)\to0$ as $s\to 0$.
\end{definition}
\begin{lemma}
Let $\alpha_1$ and $\alpha_2$ be class $\K$ functions on $[0,a)$, $\alpha_3$ and $\alpha_4$ be class $\K_\infty$ functions, and $\beta$ be a class $\KL$ function. Denote the inverse of $\alpha_i$ by $\alpha_i^{-1}$. Then
\begin{itemize}
\item $\alpha_1^{-1}$ is defined on $[0,\alpha_1(a))$ and belongs to class $\K$.
\item $\alpha_3^{-1}$ is defined on $[0,\infty)$ and belongs to class $\K_\infty$.
\item $\alpha_1\circ\alpha_2$ belongs to class $\K$
\item $\alpha_3\circ\alpha_4$ belongs to class $\K_\infty$
\item $\sigma(r,s)=\alpha_1(\beta(\alpha_2(r),s))$ belongs to class $\KL$.
\end{itemize}
\end{lemma}
These classes of functions are connected to the Lyapunov theory for autonomous systems through these Lemmas

\begin{lemma}\label{lemma:limitPosDef} Let $V\colon D\to\R$ be a continuous positive definite function defined on a domain $D\subset\R^n$ that contains the origin. Let $B_r\subset D$ for some $r>0$. Then there exist class $\K$ functions $\alpha_1$ and $\alpha_2$, defined on $[0,r)$, such that
\[
	\alpha_1(\|x\|)\leq V(x)\leq\alpha_2(\|x\|)
\]
for all $x\in B_r$. If $D=\R^n$, the functions $\alpha_1$ and $\alpha_2$ will be defined on $[0,\infty)$ and the previous inequality will hold for all $x\in\R^n$. Moreover, if $V(x)$ is radially unbounded, then $\alpha_1$ and $\alpha_2$ can be chosen to belong to class $\K_\infty$.
\end{lemma}
\begin{proof}
Corollary C.4 in \cite{Khalil}.
\end{proof} 
If $V$ is a quadratic positive definite function $V(x) = x\T Px,\ P>0$, then the previous lemma follows from the fact that
\[
	\lambda_{\min}(P)I\leq V(x)\leq \lambda_{\max}(P)I
\]
\section{Nonautonomous Systems}
Consider the nonautonomous system 
\begin{equation}\label{eq:nonautonomousSys}
	x(t+1) = f(t,x(t))
\end{equation}
starting from $x(t_0)=x_0$ at time $t_0$, with $f\colon (\Tset\times D)\to\R^n,\ \Tset=\{t_0,t_0+1,\dots\},\ 0\in D\subset\R^n$ locally Lipschitz in $x$ on $T\times D$.
 The origin is an equilibrium point for \eqref{eq:nonautonomousSys} if 
\[
	f(t,0)=0,\ \forall t\in\Tset
\]
\begin{definition}The equilibrium point $x=0$ of \eqref{eq:nonautonomousSys} is
\begin{itemize}
\item \textbf{stable} if, for each $\epsilon>0$, there is $\delta=\delta(\epsilon,t_0)>0$ such that
\begin{equation}\label{eq:stabilityNonauto}
	\|x(t_0)\|<\delta \Rightarrow\|x(t)\|<\epsilon,\ \forall t\geq t_0\geq 0
\end{equation}
\item \textbf{uniformly stable} if, for each $\epsilon>0$, there is $\delta=\delta(\epsilon)>0$, independent of $t_0$, such that  \eqref{eq:stabilityNonauto} is satisfied
\item \textbf{unstable} if it is not stable
\item \textbf{asymptotically stable} if it is stable and there is a positive constant $c=c(t_0)$ such that $x(t)\to0$ as $t\to\infty$, for all $\|x(t_0)\|<c$.
\item \textbf{uniformly asymptotically stable} if it is uniformly stable and there is a positive constant $c$, independent of $t_0$, such that $x(t)\to0$ as $t\to\infty$, for all $\|x(t_0)\|<c$ uniformly in $t_0$; that is, for each $\eta>0$, there is $T=T(\eta)>0$ such that
\[
	\|x(t)\|<\eta,\ \forall\ t>t_0+T(\eta),\ \forall \|x(t_0)\|<c
\]
\item \textbf{globally uniformly asymptotically stable} if it is uniformly stable, $\delta(\epsilon)$ can be chosen to satisfy $\lim_{\epsilon\to\infty}\delta(\epsilon)=\infty$, and, for each pair of positive numbers $\eta$ and $c$, there is $T=T(\eta,c)>0$ such that
\[
	\|x(t)\|<\eta,\ \forall\ t>t_0+T(\eta,c),\ \forall \|x(t_0)\|<c
\]
\end{itemize}
\end{definition}
\begin{lemma}[Stability definition through class $\K$ functions]
The equilibrium point $x=0$ of $x(t+1)=f(t,x)$ is
\begin{itemize}\label{lemma:alternativeDefinitionNonuto}
\item uniformly stable if and only if there exists a class $\K$ function $\alpha$ and a positive constant $c$, independent of $t_0$, such that
\[
	\|x(t)\|\leq\alpha(\|x(t_0)\|),\ \forall\ t\geq t_0\geq 0,\ \forall\|x(t_0)\|<c
\]
\item uniformly asymptotically  stable if and only if there exist a class $\KL$ function $\beta$ and a positive constant $c$, independent of $t_0$, such that
\begin{equation}\label{eq:uniformlyAsymptStableNonaut}
	\|x(t)\|\leq \beta(\|x(t_0)\|,t-t_0),\ \forall\ t\geq t_0\geq 0,\ \forall\|x(t_0)\|<c
\end{equation}
\item globally uniformly stable if and only if inequality \eqref{eq:uniformlyAsymptStableNonaut} is satisfied for any initial state $x(t_0)$.
\end{itemize}
\end{lemma}
\begin{proof}
Proof in Appendix C.6 in \cite{Khalil}.
\end{proof}
An important case for an uniformly asymptotically stable point is when $\beta(r,s)=kre^{-\lambda s}$, with $\lambda>0$. In this case we have the following 
\begin{definition} The equilibrium point $x=0$ of \eqref{eq:nonautonomousSys} is called \textbf{exponentially stable} if there exist positive constants $c,k$ and $\lambda$ such that it holds
\begin{equation}\label{eq:exponentiallyStable}
	\|x(t)\|\leq k\|x(t_0)\| e^{-\lambda (t-t_0)},\ \forall\|x(t_0)\|<c
\end{equation}
and is said to be \textbf{globally exponentially stable}  if the previous inequality holds for any initial state $x(t_0)$.
\end{definition}
Note that since $\lambda$ is positive, $e^{-\lambda (t-t_0)}$ is equivalent to $\gamma^{t-t_0}$ with $\gamma=e^{-\lambda}<1$.
\begin{theorem}[Lyapunov function implies stability for nonautonomous]\label{theo:lyapNonautStable}
Let $x=0$ be an equilibrium point for the nonautonomous system 
\[
	x(t+1) = f(t,x(t))
\]
with $f\colon (\Tset\times D)\to\R^n,\ 0\in D\subset\R^n$ locally Lipschitz in $x$ on $\Tset\times D$. Let $V\colon \Tset\times D\to\R$ be a continuous function such that 
\begin{align*}
	W_1(x)\leq V(t,x)\leq W_2(x)\\
	V(t+1,f(t,x))-V(t,x)\leq 0 
\end{align*}
for all $t\geq 0$ and for all $x\in D$, where $W_1(x)$ and $W_2(x)$ are continuous positive definite functions on $D$. Then $x=0$ is uniformly stable.
\end{theorem}
\begin{proof}
Choose $r>0$ and $c>0$ such that $B_r\subset D$ and $c<\min_{\|x\|=r}W_1(x)$. Then $\{x\in B^r\ |\ W_1(s)\leq c\}$ is in the interior of $B_r$. Define $\Omega_{t,c}$ as
\[
	\Omega_{t,c}= \{x\in B_r|V(t,x)\leq c\}
\]
The set $\Omega_{t,c}$ contains $\{x\in B^r\ |\ W_2(s)\leq c\}$, since $W_2(x)\leq c \Rightarrow V(t,x)\leq c$; for similar reasons $\Omega_{t,c}\subset \{x\in B^r\ |\ W_1(s)\leq c\}$. So we have
\[
	\{x\in B^r\ |\ W_2(s)\leq c\}\subset \Omega_{t,c} \subset \{x\in B^r\ |\ W_1(s)\leq c\} \subset B_r \subset D
\]
for all $t\geq 0$. Since $V(t+1,x)-V(t,x)\leq 0$ in $D$, for any $t_0\geq 0$ and $x(t_0)\in\Omega_{t_0,c}$, the solution starting at $(t_0,x(t_0))$ will stay in $\Omega_{t,c}$ for all $t\geq t_0$. We have shown that a solution is bounded and defined for all $t\geq t_0$. We now use Lemma \ref{lemma:alternativeDefinitionNonuto} (we still don't have sets defined on the norm of vector $x$).  Due to the second property of $V$ we have
\[
	V(t,x(t))\leq V(t_0,x(t_0)),\ \forall\ t\geq t_0
\]
Since $W_1$ and $W_2$ are positive definite matrix,  due to lemma \ref{lemma:limitPosDef} there are class $\K$ functions $\alpha_1$ and $\alpha_2$ defined in $[0,r)$ such that
\[
	\alpha_1(\|x\|)\leq W_1(x)\leq V(t,x)\leq W_2(x)\leq \alpha_2(\|x\|) \Rightarrow \alpha_1(\|x\|)\leq V(t,x)\leq \alpha_2(\|x\|)
\]
So we have (note that $\alpha_1$ is smaller than $V$, so to reach the same value of $V$, the argument of $\alpha_1$ has to be greater than the norm of the vector in V)
\[
	\|x(t)\|\leq \alpha_1^{-1}(V(t,x(t)))\leq \alpha_1^{-1}(V(t_0,x(t_0))\leq \alpha^{-1}(\alpha_2(\|x(t_0)\|))
\]
Since $\alpha_1^{-1}(\alpha_2(x))$ is a class $\K$ function we are done.
\end{proof}
\begin{theorem}[Lyapunov function for asymptotically stable nonautonomous systems]\label{theo:lyapNonautAsympt}
Suppose the assumptions of Theorem \ref{theo:lyapNonautStable} are satisfied and that it also  holds
\[
	V(t+1,f(x,t))-V(t,x)\leq -W_3(x),\ \forall\ t\geq 0, x\in D
\]
where $W_3(x)$ is a continuous  positive definite function on $D$. Then, $x=0$ is uniformly asymptotically stable. 
%
\end{theorem}
\begin{proof} Consider $r>0$ such that $B_r\subset D$. Due to theorem \ref{lemma:limitPosDef}, there exist class $\K$ functions $\alpha_1,\alpha_2,\alpha_3$ on $[0,r)$ such that
\[
	\alpha_1(\|x\|)\leq V(t,x)\leq \alpha_2(\|x\|) \quad V(t+1,f(x,t))-V(t,x)\leq -\alpha_3(\|x\|)  
\]
For any fixed $\epsilon,\ r\geq \epsilon>0,$ there exists a positive constant $\delta\leq\epsilon$ that satisfy the stability property. Consider a value $\eta,\ 0<\eta<\epsilon$ and define
$\Lambda=\alpha_1(\eta)$. Consider $x(t_0)\in B_\delta$; if $V(t_0,x(t_0))<\Lambda$, then the definition is already satisfied since $\|x(t_0)\|$ is necessarily smaller then $\eta$ and $V$ is contracting along the trajectory.\\ 
If $x(t_0)$ satisfies $V(t_0,x(t_0))\geq\Lambda$, define $\Gamma=\alpha_2^{-1}(\Lambda)>0$ (note that $\Gamma\leq \eta$). It follows that if $x$ is such that $V(t,x)\geq \Lambda$, then $\|x\|\geq \Gamma$. Let $\Omega=\{x\ |\ \Gamma\leq\|x\|\leq \epsilon\}$, which is closed and bounded. For all $t\geq t_0$ such that $V(t,x(t))$ is greater than $\Lambda$, it also holds $\|x(t)\|\geq \Gamma>0$, so $\alpha_3(\|x(t)\|)\geq\alpha_3(\Gamma)$ which implies $-\alpha_3(\|x(t)\|)\leq-\alpha_3(\Gamma)$. So we have $V(t+1,f(x,t))-V(t,x)\leq -\alpha_3(\|x\|)\leq-\alpha(\Gamma)<0$. So it holds $\forall k\geq 0\ |\ \tilde{V}(t_0+k)\geq\Lambda$
\[
	\tilde{V}(t_0+k)=\tilde{V}(t_0)+\sum_{i=0}^{k-1} \left[\tilde{V}(t_0+i+1)-\tilde{V}(t_0+i)\right]\leq \tilde{V}(t_0)-k \alpha_3(\Gamma)
\]\
This shows that there exists a $\bar{k}$, that depends on $\eta$ and on $\delta$ but not on $t_0$, such that $\Lambda\leq\tilde{V}(t_0+\bar{k})< \Lambda+\alpha_3(\Gamma)$, from which it follows that $\tilde{V}(t_0+\bar{k}+1)<\Lambda$ and as already discussed $\|x(t_0+k)\|<\eta$ for all $k\geq \bar{k}+1$. 
\end{proof}
\begin{definition}A function $V(t,x)$ is said to be
\begin{itemize}
\item \textbf{positive semidefinite} if $V(t,x)\geq 0$
\item \textbf{positive definite} if $V(t,x)\geq W_1(x)$ with $W_1$ positive definite
\item \textbf{radially unbounded} if $W_1(x)$ is so
\item \textbf{decrescent} if $V(t,x)\leq W_2(x)$ with $W_2$ positive definite.
\end{itemize}
\end{definition}
\begin{theorem}[Lyap exponentially bounded implies exponential stab]\label{theo: nonAutExponStab}
Let $x=0$ be an equilibrium point for the nonautonomous system 
\[
	x(t+1) = f(t,x(t))
\]
with $f\colon (\Tset\times D)\to\R^n,\ 0\in D\subset\R^n$ locally Lipschitz in $x$ on $\Tset\times D$. Let $V\colon \Tset\times D\to\R$ be a positive definite continuous on $x$ function such that 
\begin{gather*}
	V(t,x)< a\|x\|^2\\
	\Delta(t,x):=V(t+1,f(t,x))-V(t,x)\leq -b\|x\|^2 
\end{gather*}
for all $t\geq 0$ and for all $x\in D$, where $a$ and $b$ are positive constants. Then $x=0$ is exponentially stable. If the assumptions hold globally, then the equlibrium point is globally exponentially stable.
\end{theorem}
\begin{proof}
For any given trajectory of the system starting from $x_0\in D$, due to the assumptions it holds that
\[
	\Delta(t,x(t))\leq -(b/a) V(t,x(t)) \leq -cV(t,x(t)),\ 0<c<1
\]
Exploiting the definition of $\Delta$, we have
\[
	V(t,x(t))\leq (1-c) V(t-1,x(t-1))\leq \dots\leq (1-c)^{t-t_1}V(t_1,x(t_1)),\ t\geq t_1\geq t_0
\]
Since $c<1$, then $(1-c)^{t-t_1}=e^{-\gamma (t-t_1)},\ \gamma>0$, so 
\[
	V(t,x(t))\leq  e^{-\gamma (t-t_1)}V(t_1,x(t_1))
\]
Fixing $t_1=t_0+p$, $p$ can be chosen such that $V(t_1,x(t_1))\leq d\|x(t_0)\|$, with $d$ independent of $x(t_0)$. Now, combining the previous results,
\[
	\Delta(t,x(t))\leq -cV(t,x(t))\leq -c e^{-\gamma (t-t_1)}V(t_1,x(t_1))
\]
from which it follows
\[
	\|x(t)\|\leq -(1/b)\Delta(t,x(t))\leq \frac{c}{b} e^{-\gamma (t-t_1)}V(t_1,x(t_1))\leq \frac{dc}{b} e^{-\gamma (t-t_0)}\|x(t_0)\|e^{\gamma p}
\]
which has the form of equation \eqref{eq:exponentiallyStable}
\end{proof}

\begin{theorem}[Exp stability assures presence Lyap funct nonauton]\label{theo:expStamImpliesLyapNonaut}
Let $x=0$ be an equilibrium point for the system 
\[
	x(t+1)=f(t,x)
\] 
where $f\colon \Tset\times D\to\R^n$ is locally Lipschitz in $x$ on $\Tset\times D$, and $D=\{x|\|x\|<r\}$. If there exist positive constants $k,c,c<r/k$ and $\lambda,\lambda<1$  such that for any intial $x(t_0)$ in $B_c=\{x|\|x\|<c\}\subset D$ the equilibrium point is exponentially stable, that is
\[
	\|x(t)\|\leq k\|x(t_0)\| e^{-\lambda (t-t_0)},\ \forall\|x(t_0)\|\in B_c
\]
then there exists a  Lyapunov function $V(t,x)$ for the system. The latter satisfies the following inequalities
\[
	\begin{array}{c}
	c_1\|x\|^2\leq V(t,x) \leq c_2 \|x\|^2\\
	V(t+1,f(t,x))-V(t,x)\leq -c_3 \|x\|^2\\
	|V(t,x)-V(t,y)|\leq c_4 \|x-y\|(\|x\|+\|y\|)
	\end{array}
\]
for all $x,y\in B_\delta$ and for some positive constants $c_1,c_2,c_3$ and $c_4$.
\end{theorem}
\begin{proof} 
Let $\phi(t,t_0;x)$ be the solution of $x(t+1)=f(t,x(t))$ at time $t$ starting from $x(t_0)=x$ at time $t_0$. It holds $\phi(t_0,t_0;x)=x$. Let
\[
	V(t,x) = \sum_{k=t}^{N-1+t}\phi(k,t;x)\T\phi(k,t;x)
\]
for some integer variable $N$ to be set. Then 
\[V(t,x) = x\T x+\sum_{k=t+1}^{N-1+t}\phi(k,t;x)\T\phi(k,t;x)\geq x\T x = \|x\|\]
and on the other hand, due to the exponential stability we have
\[
	V(t,x) = \sum_{k=t}^{N-1+1}\phi(k,t;x)\T\phi(k,t;x) \leq  \sum_{\tau=t}^{N-1+t} k^2\|x\|^2 e^{-2\lambda (\tau-t)}\leq k^2\left(\frac{1-(e^{-2\lambda})^{N}}{1-e^{-2\lambda}}\right)\|x\|^2
\]
We have shown that there exists $c_1$ and $c_2$ such that
\[
	c_1\|x\|^2\leq V(t,x) \leq c_2 \|x\|^2
\]
is satisfied. Now, since $\phi(t+1+k,t+1;f(t,x)) = \phi(t+1+k,t+1;\phi(t+1,t;x))=\phi(t+1+k,t;x)$,
\begin{align*}
	V(t+1,&f(t,x))-V(t,x)=\\ 
	&= \sum_{k=t+1}^{N-1+t+1}\phi(k,t+1;f(t,x))\T\phi(k,t+1;f(t,x))-\sum_{k=t}^{N-1+t}\phi(k,t;x)\T\phi(k,t;x)=\\
	 			&=\sum_{\Delta=0}^{N-1}\phi(t+1+\Delta,t;x)\T\phi(t+1+\Delta,t;x)-\sum_{\Delta=0}^{N-1}\phi(t+\Delta,t;x)\T\phi(t+\Delta,t;x) =\\
	 			&= \phi(t+N,t;x)\T\phi(t+N,t;x)-\phi(t,t;x)\T\phi(t,t;x)\\
	 			&\leq k^2e^{-2\lambda N}\|x\|^2-\|x\|^2 = -(1-k^2e^{-2\lambda N})\|x\|^2
\end{align*}
Now we can choose $N$ big enough so that $1-k^2e^{-2\lambda N}$ is greater than $0$ and also the second property has been proven. For the third property, since $B_c$ is a compact set, function $f(t,x)$ is Lipschitz in $B_\delta$ uniformly in $t$, with a Lipschitz constant $L$, so it holds $\|f(t,x)-f(t,y)\|\leq L\|x-y\|\ \forall t\in\Tset$. Then 
\begin{align*}
	\|\phi(t+\Delta+1,t;x)-\phi(t+\Delta+1,t;y)\|&=\|f(t+\Delta,\phi(t+\Delta,t;x))-f(t+\Delta,\phi(t+\Delta,t;y))\|\\
	\leq L\|\phi(t+\Delta,t;x)-\phi(t+\Delta,t;y)\|
\end{align*}
and by induction
\[
	\|\phi(t+\Delta,t;x)-\phi(t+k,t;y)\|\leq L^\Delta\|x-y\|
\]
Proceeding as in the proof of Theorem \ref{theo:expStabImpliesPresenceofLyapAuto} we have
\begin{align*}
	|V(t,x)-V(t,y)|&=\leq \sum_{k=t}^{N-1+t}[\|\phi\T(k,t;x)\|+\|\phi\T(k,t;y)\|]L^k\|x-y\|\\
			&\leq \left[\sum_{\tau=t}^{N-1+t} ke^{-\lambda\tau} L^k\right](\|x\|+\|y\|)\|x-y\|\\
			&\leq c_4 (\|x\|+\|y\|)\|x-y\|
\end{align*}
and so we have proven the last inequality.
\end{proof}

\subsection{Linear systems and Linearisation}
Consider now the linear time variant system
\begin{equation}\label{eq:linearSystemNonAut}
	x(t+1)=A(t)x(t)
\end{equation}
which has an equilibrium point in the origin. The following theorem holds
\begin{theorem}[Lyapunov function for linear time variant syst ]\label{theo:quadraticFunctionNonauton}
Consider the system \eqref{eq:linearSystemNonAut}. If there exists a continuous, symmetric, bounded positive definite matrix $P(t), 0<p_1 I\leq P(t)\leq p_2 I,\ \forall t\geq 0$, which satisfies the equation
\begin{equation}\label{eq:lyapEqNonAut}
	A(t)\T P(t+1) A(t)-P(t)=-Q(t)
\end{equation}
with $Q(t)$ continuous, symmetric, positive definite matrix, $Q(t)\geq q_1 I >0$, then the equilibrium point $x=0$ is globally exponentially stable.
\end{theorem}
\begin{proof}
The Lyapunov function $V(t,x)=x\T P(t)x$  satisfies
\[
	p_1\|x\|^2\leq V(t,x)\leq p_2\|x\|^2	
\]
Moreover, consider the absolute difference
\begin{align*}
	V(t+1,A(t)x)-V(t,x) &= x\T A(t)\T P(t+1)A(t)x-x\T P(t) x\\
	& =  x\T [A(t)\T P(t+1)A(t)- P(t)] x = -x\T Q(t)x 
\end{align*}
Therefore it holds
\[
	V(t+1,A(t)x)-V(t,x)\leq-q_1\|x\|^2
\]
and the assumptions of theorem \ref{theo: nonAutExponStab} are satisfied.
\end{proof}
Define now the transition matrix $\Phi(t,t0)$ which is such  that $x(t) = \Phi(t,t_0)x(t_0)$. It holds $\Phi(t_0,t_0)=I$ and for linear time variant system its form is $\Phi(t,t_0)=A(t-1)\cdot  A(t-2)\cdots A(t_0)$.
\begin{theorem}[Condition on the transition matrix to have exp stab]\label{theo:conditionTransitionMatrixAsymp}
The equilibrium point $x=0$ of the linear time variant system
\[
	x(t+1)=A(t)x(t)
\]
 is uniformly asymptotically stable if and only if the state transition matrix satisfies
\begin{equation}\label{eq:transMatrixExpo}
	\|\Phi(t,t_0)\|\leq ke^{-\lambda(t-t_0)},\ \forall \ t\geq t_0\geq 0
\end{equation}
for some positive constants $k$ and $\lambda$.
\end{theorem}
\begin{proof}
We first introduce 2 lemmas
\begin{lemma} \label{lemma:transMatrUnifStable}
If system \eqref{eq:linearSystemNonAut} is uniformly stable, then there exists a constant $M$ independent of $t_0$ such that $\|\Phi(t,t_0)\|\leq M$ for all $t\geq t_0$.
\end{lemma}
\begin{proof}
If the system is stable, fixing $\epsilon$, we can choose $\delta>0$ such that if $\|x(t_0)\|\leq \delta$, then $\|\Phi(t,t_0)x(t_0)\|\leq \epsilon$ for all $t\geq t_0$. It follows 
\[
	\max_{\|x\|=\delta}\|\Phi(t,t_0)x\|=\max_{\|x\|=1}\|\Phi(t,t_0)\delta x\|=\delta\max_{\|x\|=1}\|\Phi(t,t_0)x\|\leq \epsilon
\]
so, using the induced norm,
\[
	\|\Phi(t,t_0)\|=\max_{\|x\|=1}\|\Phi(t,t_0)x\|\leq \frac{\epsilon}{\delta^{-1}}:=M
\]
\end{proof}
\begin{lemma}\label{lemma:transMatrAsymptUnifStable}
The following statements are equivalent
\begin{enumerate}
\item System \eqref{eq:linearSystemNonAut} is uniformally asymptotically stable
\item System \eqref{eq:linearSystemNonAut} is globally uniformly asymptotically stable
\item $\|\Phi(t,t_0)\|\to 0$ as $t\to\infty$ uniformly in $t_0$
\item Given $\{z_i\}_{i=1}^n$ a basis of $R^n$, then $\|\Phi(t,t_0)z_i\|\to 0$ as $t\to\infty$ uniformly in $t_0$
\end{enumerate}
\end{lemma}
\begin{proof}
$(i)\Rightarrow(ii)$ follows from linearity, and the implications $(ii)\Rightarrow(iv)$ and $(iii)\Rightarrow(i)$ can be easily verified. Concerning $(iv)\Rightarrow(iii)$ we can proceed as follows: if $(iv)$ holds, then for every $\epsilon>0$ there exists a time $\tau(\epsilon)$ independent of $t_0$ such that $\|\Phi(t, t_0)z^i\| < \epsilon$ for all $t\geq t_0+\tau(\epsilon)$ and $i=1,\dots,n$. For every $x(t_0) = \sum_{i=1}^n y_i z_i$, $\|x(t_0)\| = 1$, there exists a positive constant $a$ such that $\max |y_i|\leq a^{-1}$. Thus
\[
	\|\Phi(t,t_0)x(t_0)\|=\left\Vert\sum_{i=1}^n y_i\Phi(t,t_0) z_i\right\Vert\leq a^{-1}n\epsilon,\ t\geq t_0+\tau(\epsilon)
\]
and again, due to the induced norm, this proves $(iii)$.
\end{proof}
Going back to the proof of the theorem, if the transition matrix satisfies \eqref{eq:transMatrixExpo}, then due to Lemma \ref{lemma:transMatrAsymptUnifStable} the system is uniformly asymptotically stable. On the other hand suppose that the system is uniformly asymptotically stable. By Lemma \ref{lemma:transMatrAsymptUnifStable}, there exists $\tau\geq 0$ such that $\|\Phi(t+\tau,t)\|\leq 1/2$ for all $t\geq t_0$. It follows
\[
	\|\Phi(t_0+k\tau,t_0)\|\leq  \|\Phi(t_0+k\tau,t_0+(k-1)\tau)\|\dots \|\Phi(t_0+\tau,t_0)\|\leq 2^{-k}
\]
Now suppose $t_0+k\tau\leq t<t_0+(k+1)\tau,\ t\geq t_0,\ k\in\N$, then
\[
	\|\Phi(t,t_0)\|\leq \|\Phi(t,t_0+k\tau)\|\|\Phi(t_0+k\tau,t_0)\|\leq \|\Phi(t,t_0+k\tau)\|2^{-k}
\]
Now, due to Lemma \ref{lemma:transMatrUnifStable}, there exists a constant $M'$ such that $\|\Phi(t,t_0+k\tau)\|\leq M'$ for all $t\geq t_0+k\tau,\ k\in\N$ and so
\[
	\|\Phi(t,t_0)\|\leq M'2^{-[(t-t_0)/\tau-1]},\ t\geq t_0
\]
Choosing $k=2M'$ and $\lambda=-1/(\tau)\log_e(2)$ the theorem is proved.
\end{proof}
This theorem show that uniform asymptotic stability is equivalent to exponential stability.
\begin{theorem}[An exp stable lin syst has a Lyap funct]\label{theo:expStabLinearHasLyapFunct}
Let $x=0$ be the exponentially stable equilibrium point of the linear time variant system
\[
	x(t+1)=A(t)x(t)
\]
and suppose that $A(t)$ is bounded. Let $Q(t)$ be a bounded, positive definite, symmetric matrix, i.e. $0<q_1 I\leq Q(t)\leq q_2 I$. Then, there is a bounded, positive definite, symmetric matrix $P(t)$, i.e. $0<p_1 I\leq P(t)\leq p_2 I$, that satisfies \eqref{eq:lyapEqNonAut}. Hence $V(t,x)=x\T P(t)x$ is a Lyapunov function for the system, that also satisfies the conditions of Theorem \ref{theo: nonAutExponStab}.
\end{theorem}
\begin{proof}
Let 
\[
	P(t)=\sum_{\tau=t}^\infty \Phi(\tau,t)\T Q(\tau)\Phi(\tau,t)
\]
Therefore we have
\[
	V(t,x)=x\T P(t)x = \sum_{\tau=t}^\infty x\T \Phi(\tau,t)\T Q(\tau)\Phi(\tau,t)x\leq q_2 \sum_{\tau}\|\Phi(\tau,t) x\|^2
\]
Using theorem \ref{theo:conditionTransitionMatrixAsymp}, we have
\[
	V(t,x)\leq q_2\|x\|^2\sum_{\tau=t}^\infty k^2 e^{-2\lambda(\tau-t)}=\frac{q_2 k^2}{1-e^{-2\lambda}}\|x\|^2\leq p_1\|x\|^2
\]
On the other hand
\[
	V(t,x)\geq q_1 \sum_{\tau=t}^\infty x\T \Phi(\tau,t)\T \Phi(\tau,t)x\geq q_1\|x\|^2
\]
considering only the first element of the summation. So we have
\[
	 q_1\|x\|^2\leq  V(t,x)\leq p_1\|x\|^2 \Rightarrow q_1 I\leq P(t)\leq p_1 I
\]
and so $P(t)$ is positive definite and bounded.\\
Now let us check whether \eqref{eq:lyapEqNonAut} is satisfied, so let us evaluate
\begin{align*}
	A(t)\T P(t+1)A(t)-P(t)& = A(t)\T \sum_{\tau=t+1}^\infty \left[\Phi(\tau,t+1)\T Q(\tau)\Phi(\tau,t+1)\right]A(t)-\\
			&-\sum_{\tau=t}^\infty \Phi(\tau,t)\T Q(\tau)\Phi(\tau,t)
\end{align*}
Now since $\Phi(\tau,t+1)A(t)=\Phi(\tau,t)$, it holds
\[
	\sum_{\tau=t+1}^\infty \Phi(\tau,t)\T Q(\tau)\Phi(\tau,t)-\sum_{\tau=t}^\infty \Phi(\tau,t)\T Q(\tau)\Phi(\tau,t)=-\Phi(t,t)\T Q(t)\Phi(t,t)=-Q(t)
\]
so \eqref{eq:lyapEqNonAut} is satisfied. From the latter we have
\[
	V(t+1,A(t)x)-V(t,x)=-x\T Q(t)x\leq -q_1\|x\|^2
\]
and so $V(t,x)$ satisfies all the assumptions of Theorem \ref{theo: nonAutExponStab}
\end{proof}
Now, using the Lyapunov function for the linear system we will prove some linearisation results. Consider again the general nonlinear nonautonomous system
\[
	x(t+1)=f(t,x)
\]
where $f\colon \Tset\times D\to\R^n$ is locally Lipschitz in $x$ on $\Tset\times D$, and $D=\{x\in \R^n\ | \|x\|<r\}$. Suppose that $f(t,0)=0,\ \forall t\in \Tset$, that is $x=0$ is an equilibrium point for the system. Moreover suppose that the Jacobian matrix $[\partial f/\partial x]$ is bounded and Lipschitz on $D$, from which it follows, for all $i=1,\dots,n$
\[
	\left\Vert \frac{\partial f_i}{\partial x}(t,x_1)-\frac{\partial f_i}{\partial x}(t,x_2)\right\Vert_2 \leq L_1\|x_1-x_2\|_2,\ \forall x_1,x_2\in D,\ \forall t\in\Tset
\]
By the mean value theorem, there exists a $z_i\in D$ on the line segment between the origin and $x\in D$ such that
\[
	f_i(t,x)=f_i(t,0)+\frac{\partial f_i}{\partial x}(t,z_i)x
\]
Since $f(t,0)=0$, $f_i(t,x)$ can be rewritten as
\[
	f_i(t,x)= \frac{\partial f_i}{\partial x}(t,0)x+\underbrace{\left[\frac{\partial f_i}{\partial x}(t,z_1)-\frac{\partial f_i}{\partial x}(t,0)\right]x}_{g_i(t,x)}
\]
Defining $A(t) = \frac{\partial f}{\partial x}(t,0),\ f(t,x)$ can be rewritten as
\begin{equation}\label{eq:nonautSystemLin+nonlin}
	f(t,x)=A(t)x+g(t,x)
\end{equation}
The nonlinear part is bounded in norm, since
\begin{align}
\|g(t,x)\|_2 & \leq \left(\sum_{i=1}^n \left\Vert \frac{\partial f_i}{\partial x}(t,z_1)-\frac{\partial f_i}{\partial x}(t,0)\right\Vert_2^2\right)^{1/2} \|x\|_2\leq\nonumber\\
&\leq  \left(\sum_{i=1}^n L_1^2\underbrace{\|z_i\|^2}_{\leq \|x\|^2}\right)^{1/2} \|x\|_2\leq \underbrace{\sqrt{n}L_1}_{L} \|x\|_2^2 \label{eq:boundgnonaut}
\end{align}
This implies that in a neighbourhood of the origin we can approximate the nonlinear function $f(t,x)$ with its linearisation $A(t)x$. We can therefore apply the Lyapunov function found for the linearised system to the starting nonlinear system.

\begin{theorem}[If lin system is exp stable than the nonlin is exp stabl]\label{theo:linSystExpStableImpliesnonlinExpStableNonaut}
Let $x=0$ be an equilibrium point for the nonlinear system
\[
	x(t+1)=f(t,x)
\]
where $f\colon \Tset\times D\to\R^n$ is locally Lipschitz in $x$ on $\Tset\times D$, and $D=\{x\in \R^n\ | \|x\|<r\}$. Suppose that the Jacobian matrix $[\frac{\partial f}{\partial x}]$ is bounded and Lipschitz on $D$, uniformly in $t$. Let
\[
	A(t) = \left.\frac{\partial f}{\partial x}(t,x)\right\vert_{x=0}
\]
Then the origin is an exponentially stable equilibrium point for the nonlinear system if it is an exponentially stable equilibrium point for the linear system $x(t+1)=A(t)x(t)$.
\end{theorem}
\begin{proof}
From the assumptions we have that $\|A(t)\|\leq B_A$. Due to theorem \ref{theo:expStabLinearHasLyapFunct}, given bounded and positive definite matrices $Q(t),\ t\in\Tset$ there exist bounded and positive definite matrices matrices $P(t)$ such that $V(t,x)= x\T P(t)x$ is a Lyapunov function for the linearised system system. The matrices $P(t)$ and $Q(t)$ satisfy the following inequalities
\[
	0<p_1 I\leq P\leq p_2,\quad 0<q_1 I\leq Q\leq q_2 I
\]
Let us use function $V(x,t)$ for the nonlinear system. To prove that it is a Lyapunov function also for the nonlinear system we have to check whether the absolute difference  $V(t+1,f(t,x))-V(t,x)$ is negative definite. Using the rewriting \eqref{eq:nonautSystemLin+nonlin} for function $f$, we have the following
{\small
\begin{align*}
	V(&t+1,f(t,x))-V(t,x) = (x\T A(t)\T +g(t,x)\T) P(t+1) (A(t)x+g(t,x))-x\T P(t) x=\\
	&=x\T(A(t)\T P(t+1)A(t)-P(t))x+2g(x,t)\T P(t+1)A(t) x+g(t,x)\T P(t+1)g(t,x)	=\\
	&=x\T(-Q(t))x+2g(x,t)\T P(t+1)A(t) x+g(t,x)\T P(t+1)g(t,x)\leq\\
	&\leq -q_1\|x\|_2^2+2p_1LB_A\|x\|_2^3+p_1L^2\|x\|_2^4= (-q_1+2p_1LB_A\|x\|_2+p_1L^2\|x\|_2^2)\|x\|_2^2
\end{align*}
}
For the latter to be negative definite, the term $-q_1+2p_1LB_A\|x\|_2+p_1L^2\|x\|_2^2$ has to be negative. As in the autonomous case, this is a parabola directed upward and with the vertex in the third quarter, so there exists $\bar{\delta}>0$ such that as long as $\|x\|=\bar{\delta}$ then $V(t+1,f(t,x))-V(t,x)=0$. Choosing $\delta<\bar{\delta},\ \delta<r$, and defining the set $B_\delta=\{x|\|x\|\leq \delta\}$, if $x\in B_\delta$ then $V(t,x)$ is a Lyapunov function for the nonlinear function. 
\end{proof}

\begin{corollary}If the assumptions of Theorem \ref{theo:linSystExpStableImpliesnonlinExpStableNonaut} are satisfied, there exists a Lyapunov function $V(t,x)$ for the nonlinear system defined in $\Tset\times B_\delta$ that satisfies the following inequalities
\[
	\begin{array}{c}
	c_1\|x\|^2\leq V(t,x) \leq c_2 \|x\|^2\\
	V(t+1,f(t,x))-V(t,x)\leq -c_3 \|x\|^2\\
	|V(t,x)-V(t,y)|\leq c_4 \|x-y\|(\|x\|+\|y\|)
	\end{array}
\]
for all $x,y\in B_\delta$ and for some positive constants $c_1,c_2,c_3$ and $c_4$.
\end{corollary}
\begin{proof}
Due to the assumptions of the Theorem \ref{theo:linSystExpStableImpliesnonlinExpStableNonaut}, the nonlinear system is exponentially stable for $x(t_0)\in B_\delta$, so the assumptions of Theorem \ref{theo:expStamImpliesLyapNonaut} are satisfied. 
\end{proof}

\section{Inverse Lyapunov Theorems}

{\color{black}
\begin{proposition}[finite time convergence]\label{thm:inv_lyap_finite_conv}
Let us consider the following dynamical systems:
\begin{equation}\label{eq:fast2} 
y(k+1)=\varphi(k,y(k),x), \ \ y(\bar{k})=y_{\bar k}, \ \ k\geq \bar{k}, \ \ \bar{k}\in\mathbb N
\end{equation}
where $\varphi(\cdot)$ is \textbf{continuously differentiable} in $x$ and $y$ and \textbf{ globally uniformly Liptschitz in $y$}, i.e.
$$ \|\varphi(k,y_1,x)-\varphi(k,y_2,x)\|\leq L_1 \|y_1-y_2\|$$
Let us assume that there exists a \textbf{continuously differentiable} function $y^*(x)$ such that
$$ y^*(x)=\varphi(k,y^*(x),x) \ \ \forall k, x$$ 
Let us define the following operator:
$$ \varphi'(k,y',x):= \varphi(k,y'+y^*(x),x)-y^*(x) $$
and assume that \textbf{globally} satisfies
\begin{equation}\label{eq:diff_bounded} 
\left\|\frac{\partial \varphi'(k,y',x)}{\partial x}\right\| \leq L_2 \|y'\| 
\end{equation}
Consider now the dynamical system
\begin{equation}\label{eq:fast22} 
y'(k+1)=\varphi'(k,y'(k),x), \ \ y'(\bar{k})=y'_{\bar k}, \ \ k\geq \bar{k}, \ \ \bar{k}\in\mathbb N
\end{equation}
Let moreover assume that there exist $T\geq 1$ independent of $\bar{k},y'_{\bar k}, x$ such that
$${y}'(\overline k+T)=0, \,\, {y}'(\bar{k})=y_{\bar k}-y^*(x)   , \ \ \forall (\bar{k},y_{\bar k}, x)$$
where $y'(k):= y(k)-y^*(x)$
Then there exists a function $W(k,y,x)$ and positive constants $a_1,a_2,a_3,a_4$ such that the following properties hold {\color{black} globally}:
\begin{eqnarray}
&&a_1 \|y'\|^2 \leq W(k, y', x)  \leq a_2 \|y'\|^2; \label{1}\\
&& W(k+1, \varphi(k,y'+y^*(x), x)-y^*(x),x) -  W(k, {y}', x)  \leq -a_3 \|{y}'\|^2 \ \ \ \ \ \ \  \label{2}\\
&& | W(k, {y}'_1, x) -  W(k, {y}'_2, x)| \leq a_4 \|{y}'_1 - {y}'_2\| \left(\|{y}'_1\| + \|{y}'_2\|\right)  \label{3}\\
&&| W(k, {y}', x_1) -  W(k, {y}', x_2)| \leq a_5 \|{y}'\|^2  \|x_1-x_2\| \label{4}
\end{eqnarray}
\end{proposition}
\begin{proof}
{\color{black}
The proof follows similarly to Theorem~5.8 by defining 
$$ W(k,  y', x) = \sum_{t=k}^{T-1+k} \|\psi(t; k, y',x)\|^2$$
where $\psi$ is the difference between the solution of \eqref{eq:fast2} starting from initial conditions $\bar k=k, y_{\bar k}= y'+y^*(x)$, and $y^*(x)$, i.e.
$$ \psi(t+1; k, y',x) = \varphi'(t,\psi(t; k, y',x),x), \ \ \ \psi(k; k, y',x)=y' $$
We start by proving some preliminary properties on the operator $\varphi'(k,y',x)$. The first is 
\begin{eqnarray}
\|\varphi'(k,y',x)\|&=& \|\varphi(k, y'+y^*(x),x)- y^*(x)\| \nonumber \\
&=& \|\varphi(k, y'+y^*(x),x)- \varphi(k, y^*(x),x)\| \leq L_1 \|y'\| \label{33}
\end{eqnarray}
and the second is 
\begin{eqnarray}
\|\varphi'(k,y'_1,x)-\varphi'(k,y'_2,x)\|&=& \|\varphi(k, y_1'+y^*(x),x)- \varphi(k, y_2'+y^*(x),x)\| \nonumber \\
&\leq& L_1 \|y_1'-y_2'\|
\end{eqnarray}
where we used the Liptschitz property of the operator $\varphi()$.

The operator $ \varphi'(t,y,x)$ is continuously differentiable in $x$ since it is the composition of continuously differentiable map, namely $\varphi(t,y,x)$ is continuously differentiable in both $y$ and $x$ and $y^*(x)$ is continuously differentiable in $x$.  The last property that we will use is
 \begin{eqnarray}
&&\|\varphi'(k,y',x_1)-\varphi'(k,y',x_2)\| =\\
&&=\left\| \left( \int_{0}^1 \frac{\partial \varphi'(t; k, y' ,x_1 +\eta (x_2-x_1))}{\partial x} d\eta \right)(x_1-x_2)\right\| \nonumber \\
&&\leq  \left( \int_{0}^1\left\| \frac{\partial \varphi'(t; k, y' ,x_1 +\eta (x_2-x_1))}{\partial x}\right\| d\eta \right)\|x_1-x_2\|\nonumber \\
&&\leq \int_{0}^1  L_2 \|y'\|d\eta  \|x_1-x_2\| = L_2 \|y'\|  \|x_1-x_2\| 
\end{eqnarray} 
where we used the mean value theorem for vector-valued functions.
 
Clearly
$$ W(k,  y', x) \geq \|\psi(k; k, y',x)\|^2= \| y'\|^2= a_1\| y'\|^2, \ \ a_1 =1  $$ 

Since by \eqref{33} we have that
$$ \|\psi(t+1; k, y',x)\| =\| \varphi'(t,\psi(t; k, y',x),x)\| \leq L_1 \|\psi(t; k, y',x)\| $$
this implies by induction that
$$ \|\psi(t; k, y',x)\| \leq L_1^{t-k}\|y'\| $$
therefore 
$$ W(k,  y', x) \leq \sum_{t=k}^{T-1+k} L_1^{2(t-k)}\|y'\|^2=a_2 \|y'\|^2, \ \ a_2=\sum_{t=0}^{T-1} L_1^{2k}$$
As for the second inequality, it is easy to see that
\begin{eqnarray*} &&W(k+1, \varphi(k,y'+y^*(k,x), x)-y^*(k+1,x),x) -  W(k, {y}', x) \\
&&=  W(k+1,\psi(k+1;k,y',x),x)-  W(k, {y}', x)\\
&=& \|\psi(k+T; k, y',x)\|^2 -  \|\psi(k; k, y',x)\|^2 = -\|y'\|^2 = -a_3 \|y'\|^2, \ \ \ a_3 = 1 
\end{eqnarray*}
since by assumption $\psi(k+T; k, y',x)=0$
As for the third inequality first note that
$$ |\|z_1\|^2-\|z_2\|^2|= |(z_1+z_2)^T(z_1-z_2)|\leq \|z_1+z_2\|\|z_1-z_2\|\leq (\|z_1\|+\|z_2\|)\|z_1-z_2\| $$
Moreover,  
\begin{eqnarray*} 
\|\psi(t+1; k, y_1',x)-\psi(t+1; k, y_2',x)\|&=&  \| \varphi'(t,\psi(t; k, y_1',x),x)- \varphi'(t,\psi(t; k, y_2',x),x)   \|   \nonumber  \\
 &\leq& L_1\| \psi(t; k, y_1',x)- \psi(t; k, y_2',x)\|
\end{eqnarray*}
therefore by induction
$$ \| \psi(t; k, y_1',x)- \psi(t; k, y_2',x)\| \leq L_1^{t-k}\|\psi(k; k, y_1',x)- \psi(k; k, y_2',x)\| \leq L_1^{t-k}\|y_1'-y_2'\| $$
Putting all together we obtain:
$$| \|\psi(t; k, y_1',x)\|^2- \|\psi(t; k, y_2',x)\|^2 | \leq L_1^{2(t-k)}(\|y_1'\|+\|y_2'\|)\|y_1'-y_2'\| $$
and consequently 
$$  | W(k, {y}'_1, x) -  W(k, {y}'_2, x)| \leq \underbrace{\sum_{t=k}^{T-1+k} L_1^{2(t-k)}}_{a_4} \|{y}'_1 - {y}'_2\| \left(\|{y}'_1\| + \|{y}'_2\|\right) $$
We finally have:
\begin{eqnarray*} 
&&\|\psi(t+1; k, y',x_1)-\psi(t+1; k, y',x_2)\|\\
&&=  \| \varphi'(t,\psi(t; k, y',x_1),x_1)- \varphi'(t,\psi(t; k, y',x_2),x_2)   \|   \nonumber  \\
 && =  \| \varphi'(t,\psi(t; k, y',x_1),x_1)-\varphi'(t,\psi(t; k, y',x_2),x_1)+\\ 
 &&+\varphi'(t,\psi(t; k, y',x_2),x_1)- \varphi'(t,\psi(t; k, y',x_2),x_2)   \|   \nonumber  \\
 && \leq L_1 \|\psi(t; k, y',x_1)- \psi(t; k, y',x_2)\|+L_2 \| \psi(t; k, y',x_2)\|\|x_1-x_2\|\\
 && \leq L_1 \|\psi(t; k, y',x_1)- \psi(t; k, y',x_2)\|+ L_2 L_1 \|y'\|\|x_1-x_2\|
\end{eqnarray*}
which implies by induction that
$$ \|\psi(k; k, y',x_1)-\psi(k; k, y',x_2)\|=\|y'-y'\|=0 $$
$$ \|\psi(t; k, y',x_1)-\psi(t; k, y',x_2)\| \leq L_2 \sum_{k'=1}^{t-k}L_1^{k'} \|y'\|\|x_1-x_2\|, t\geq k+1 $$
Finally
$$| \|\psi(k; k, y',x_1)\|^2- \|\psi(k; k, y',x_2)\|^2 | =| \|y'\|^2-\|y'\|^2|=0$$
$$| \|\psi(t; k, y',x_1)\|^2- \|\psi(t; k, y',x_2)\|^2 | \leq 2L_2L_1^{(t-k)}\sum_{k'=1}^{t-k}L_1^{k'} \|y'\|^2 \|x_1'-x_2'\|, \ \ t\geq k+1 $$
and consequently 
$$  | W(k, {y}', x_1) -  W(k, {y}', x_2)| \leq \underbrace{\sum_{t'=1}^{T-1} 2L_2L_1^{t'}\left(\sum_{k'=1}^{t'}L_1^{k'}\right)}_{a_5} \|y'\|^2\|x_1 - x_2\| $$

}
\end{proof}

\begin{lemma}[local result]
Let us consider the same assumptions of the previous Theorem~\ref{thm:inv_lyap_finite_conv} except for global condition~\eqref{eq:diff_bounded} is replaced with $\frac{\partial \phi'(k,y',x)}{\partial x}$ being {\bf continuously differentiable} in $y'$. Then conditions~\eqref{1}-\eqref{3} are still globally satisfied, while condition~\eqref{4} is satisfied for any $(x,y')\in \mathcal{D}$ where $\mathcal{D}\subset \mathbb{R}^{n+m}$ is a compact set that contain the origin, i.e. $(0,0)\in\mathcal{D}$ and the constant $a_5$ is possibly a function of the set $\mathcal{D}$.
\end{lemma}
\begin{proof}
First of all, note that the operator $\varphi(k,y',x)$ is continously differentiable everywhere. Let $\mathcal{D}_c$ the smallest convex compact set such that $\mathcal{D}\subseteq \mathcal{D}_c$. For any  $(x_1,y'),(x_2,y')\in \mathcal{D}$ then the point $( x_1+\eta(x_2-x_1),\lambda y')\in \mathcal{D}_c$ for any $\eta,\lambda\in [0,1]$, being  $\mathcal{D}_c$ convex and including the origin. As a result, 
$$ \|\varphi'(k,y',x_1)-\varphi'(k,y',x_2)\| \leq  \left( \int_{0}^1\left\| \frac{\partial \varphi'(t; k, y' ,x_1 +\eta (x_2-x_1))}{\partial x}\right\| d\eta \right)\|x_1-x_2\|$$ 
is still valid. Since $\frac{\partial^2 \phi'(k,0,x)}{\partial x\partial y'}=0$ and being continuous we can apply the mean value theorem to 
$$\frac{\partial \varphi'(t; k, y' ,x_1 +\eta (x_2-x_1))}{\partial x}=\left( \int_{0}^1 \frac{\partial^2 \varphi'(t; k, \lambda y' ,x_1 +\eta (x_2-x_1))}{\partial x\partial y'} d\lambda \right) y'$$ 
Finally, since $\mathcal{D}_c$ is compact and $\frac{\partial^2 \phi'(k,y',x)}{\partial x\partial y'}$ is continuous, then there exists $L_2$ possibly function of $\mathcal{D}_c$ ( and therefore of $\mathcal{D}$), such that
$$\left\| \frac{\varphi'(t; k, \lambda y' ,x_1 +\eta (x_2-x_1))}{\partial x\partial y'}\right\|\leq L_2 $$
from which the final claim of the lemma follows. 
\end{proof}

\begin{proposition}[exponential convergence]\label{thm:inv_layp_exp}
Let us consider the following dynamical systems:
\begin{equation}\label{eq:fast2} 
y(k+1)=\varphi(k,y(k),x), \ \ y(\bar{k})=y_{\bar k}, \ \ k\geq \bar{k}, \ \ \bar{k}\in\mathbb N
\end{equation}
where $\varphi(\cdot)$ is \textbf{continuously differentiable} in $x$ and $y$ and \textbf{ globally uniformly Liptschitz in $y$}, i.e.
$$ \|\varphi(k,y_1,x)-\varphi(k,y_2,x)\|\leq L_1 \|y_1-y_2\|$$
Let us assume that there exists a \textbf{continuously differentiable} function $y^*(x)$ such that
$$ y^*(x)=\varphi(k,y^*(x),x) \ \ \forall k, x$$ 
Let us define the following operator:
$$ \varphi'(k,y',x):= \varphi(k,y'+y^*(x),x)-y^*(x) $$
and assume that it \textbf{globally} satisfies
$$ \left\|\frac{\partial \varphi'(k,y',x)}{\partial x}\right\| \leq L_2 \|y'\| $$
Consider now the dynamical system
\begin{equation}\label{eq:fast22} 
y'(k+1)=\varphi'(k,y'(k),x), \ \ y'(\bar{k})=y'_{\bar k}, \ \ k\geq \bar{k}, \ \ \bar{k}\in\mathbb N
\end{equation}
Let moreover assume that there exist exist $C,\rho$ independent of $\bar{k},y_{\bar k}, x$ independent of $\bar{k},y'_{\bar k}, x$ such that
$$\|{y}'(k)\| \leq C \rho^{k-\bar{k}}\,\, \| {y}'(\bar{k})\|=\| y_{\bar k}-y^*(x,\bar k)   \|, \ \ \forall (\bar{k},y_{\bar k}, x)$$
where $y'(k):= y(k)-y^*(k,x)$
Then there exists a function $W(k,y,x)$ and positive constants $a_1,a_2,a_3,a_4$ such that the following properties hold {\color{black} globally}:
\begin{eqnarray}
&&a_1 \|y'\|^2 \leq W(k, y', x)  \leq a_2 \|y'\|^2; \label{11}\\
&& W(k+1, \varphi(k,y'+y^*(x), x)-y^*(x),x) -  W(k, {y}', x)  \leq -a_3 \|{y}'\|^2 \ \ \ \ \ \ \  \label{22}\\
&& | W(k, {y}'_1, x) -  W(k, {y}'_2, x)| \leq a_4 \|{y}'_1 - {y}'_2\| \left(\|{y}'_1\| + \|{y}'_2\|\right)  \label{33}\\
&&| W(k, {y}', x_1) -  W(k, {y}', x_2)| \leq a_5 \|{y}'\|^2  \|x_1-x_2\| \label{44}
\end{eqnarray}
\end{proposition}

{\color{black}

\begin{proof}
{\color{black}
The proof follows similarly to Theorem~\ref{thm:inv_lyap_finite_conv} by defining 
$$ W(k,  y', x) = \sum_{t=k}^{T-1+k} \|\psi(t; k, y',x)\|^2$$
where $T$ is to designed. This will be obtained by first guaranteeing inequality~\eqref{22}:
\begin{eqnarray*} &&W(k+1, \varphi(k,y'+y^*(k,x), x)-y^*(k+1,x),x) -  W(k, {y}', x) \\
&&=  W(k+1,\psi(k+1;k,y',x),x)-  W(k, {y}', x)\\
&=& \|\psi(k+T; k, y',x)\|^2 -  \|\psi(k; k, y',x)\|^2 \leq C\rho^T \|y'\|^2 -\|y'\|^2 
\end{eqnarray*}
If we now pick 
$$ T=T^* :=  \lceil -\frac{\log(2C)}{\log \rho} \rceil $$
where $\lceil \cdot \rceil$ indicates the smallest integer number greater or equal to the argument, then inequality~\eqref{22} is satisfied with $a_3 = \frac{1}{2}$. All the other constants $a_1,a_2,a_4,a_5$ are the same as those in Theorem~\ref{thm:inv_lyap_finite_conv} just by substituting $T$ with $T^*$.
}
\end{proof}

}

%

\section{Averaging}

%
%

\begin{proposition}\label{prop:averaging}
Let us consider the dynamical system:
\begin{equation}\label{eqn:slow}
x(k+1) = x(k)+ \epsilon \phi(k,x(k))
\end{equation}
 and assume that $\phi$ is {\bf globally uniformly Liptschitz}, i.e.
$$ \|\phi(k,x)\|\leq L \|x\| $$
Also assume that the following limit exists 
$$ \overline \phi(x) = \lim_{T \to \infty }\frac{1}{T} \sum_{k'=k}^{k+T} \phi(k,x), \ \ \forall k, x $$
and satisfies the following condition
$$ \left\|  \sum_{k'=k}^{k+T} \phi(k,x) - T  \overline \phi(x)\right\| \leq T \sigma(T) L \|x\| $$
where $\sigma(T)$ is a monotonically function decreasing to zero, i.e. 
$$ \lim_{T\to \infty} \sigma(T)=0$$
Then, there exists a function $g'(k,T,x(k),\epsilon)$ such that
$$x(k+T+1) = x(k)+\epsilon T \overline \phi(x(k)) + g'(k,T,x(k),\epsilon), \ \ T\geq 0, \epsilon \in[0,1] $$
where 
$$ \| g'(k,T,x(k),\epsilon) \| \leq \epsilon T\nu(T,\epsilon) \|x(k)\|$$
$$ \nu(T,\epsilon):=L\sigma(T)+ \epsilon L^2 T (1+L)^T$$

\end{proposition}
\begin{proof}
We will prove the previous results by intermediate steps
\begin{itemize}
\item We first show that
$$x(k+T+1) = x(k)+\epsilon \left(\sum_{k'=k}^{k+T} \phi(k',x(k))  \right) + g(k,T,x(k),\epsilon), \ \ T\geq 0 $$
where
$$ \|g(k,T,x(k),\epsilon)\| \leq M_{T} \|x(k)\|, \ \ T\geq 0  $$
$$ M_{T+1}=(1+\epsilon L) M_T + \epsilon^2 L^2 T, \ \ M_0=0 $$
Let start by defining
$$ g(k,T,x(k),\epsilon):= x(k+T+1) - x(k)-\epsilon \left(\sum_{k'=k}^{k+T} \phi(k',x(k))  \right)$$
such function is obviously well defined.
We will prove the previous inequality by induction. For $T=0$ we have $g(k,0,x(k),\epsilon)=0$, therefore $M_0=0$.
Let us assume that the inequality is true for $T$ and let us prove it holds also for $T+1$. Let us first define
$$h(k,T,x(k))=\sum_{k'=k}^{k+T} \phi(k',x(k))$$
Clearly
$$\| h(k,T,x(k))\| \leq T L\|x(k)\|$$
Therefore
\begin{eqnarray*}
g(k,T+1,x(k),\epsilon) &=& x(k+T+2)- x(k)-\epsilon \left(\sum_{k'=k}^{k+T+1} \phi(k',x(k))  \right)\\
&=& x(k+T+1)+ \epsilon \phi(k+T+1,x(k+T+1)) -x(k) \\
&&- \epsilon \phi(k+T+1,x(k)) -\epsilon \left(\sum_{k'=k}^{k+T} \phi(k',x(k))\right)\\
&=& g(k,T,x(k),\epsilon)+\epsilon \big(\phi(k+T+1,x(k+T+1))-\phi(k+T+1,x(k)) \big) 
\end{eqnarray*}
Also note that by definition
$$ x(k+T+1) = x(k)+ \epsilon h(k,T,x(k))+g(k,T,x(k),\epsilon)$$
Therefore we get:
\begin{eqnarray*}
M_{T+1}&=&\|g(k,T+1,x(k),\epsilon)\| \\
&\leq& \|g(k,T,x(k),\epsilon)\|+\epsilon L \| x(k+T+1)-x(k)\| \\
&\leq & M_T +\epsilon L ( \epsilon \|h(k,T,x(k))\|+\|g(k,T,x(k),\epsilon)\|)\\
&\leq & M_T + \epsilon^2L^2 T \|x(k)\|+\epsilon L M_T \\
&=& (1+\epsilon L) M_T+\epsilon^2 T L^2 \|x(k)\|
\end{eqnarray*}
%
\item Note that the previous sequential inequality which is a simple linear system with a time-varying input can be used to show that
$$ M_T \leq  \epsilon^2 L^2 \|x(k)\| \sum_{k'=0}^{T} (1+\epsilon L)^{T-k'}k'  $$
By definition we can immediately observe that
$$g'(k,T,x(k),\epsilon) = g(k,T,x(k),\epsilon) + \epsilon \left(\sum_{k'=k}^{k+T} \phi(k',x(k))  - T \overline \phi(x(k)) \right)$$
which implies that
 $$ \|g'(k,T,x(k),\epsilon)\| \leq \|g(k,T,x(k),\epsilon)\|+\epsilon L T \sigma(T) \|x(k)\| $$
where we used the property of the averaged operator $\overline \phi$. Consequently
$$ \|g'(k,T,x(k),\epsilon)\| \leq \left( \epsilon T L\sigma(T) + \epsilon^2 L^2 \sum_{k'=0}^{T} (1+\epsilon L)^{T-k'}k' \right)\|x(k)\| $$
\item Now note that for $\epsilon \in [0,1]$ we have
$$ \sum_{k'=0}^{T} (1+\epsilon L)^{T-k'}k' \leq \sum_{k'=0}^{T} (1+L)^{T-k'}k'\leq \sum_{k'=1}^{T} (1+L)^{T-k'}T\leq \sum_{k'=1}^{T} (1+L)^{T}T = T^2(1+L)^T$$
therefore 
$$ \|g'(k,T,x(k),\epsilon)\| \leq \epsilon T \left( L\sigma(T)+\epsilon L^2 T (1+L)^T \right)\|x(k)\| $$
Clearly by defining
$$ \nu(T,\epsilon):=L\sigma(T)+ \epsilon L^2 T (1+L)^T$$
we obtain the claim.

\end{itemize}
\end{proof}

\begin{proposition}\label{prop:bound}
Let us consider the following function
$$ \mu(T,\epsilon):= \nu(T,\epsilon)+\epsilon T (L+\nu(T,\epsilon) )^2$$ where the function $\nu(T,\epsilon)$ has been defined in the previous theorem. Then, for any $\delta>0$, there exist $\epsilon_\delta\in(0,1]$ and $T_\delta\in\mathbb N$ such that 
$$\mu(T_\delta,\epsilon)\leq \delta, \ \ \ \forall \epsilon\in [0,\epsilon_\delta] $$
\end{proposition}
\begin{proof}
We will prove the previous claim by construction. Let
$$ T_\delta := \min\{ T\in\mathbb N \, | \, \sigma'(T)\leq \frac{\delta}{4}\} $$
and
$$ \epsilon_1:= \frac{\delta}{4L^2T_\delta(1+L)^{T_\delta}}$$
This suffices to claim that
$$\nu(T_\delta,\epsilon) \leq \nu(T_\delta,\epsilon_1) \leq  \frac{\delta}{2}, \ \ \epsilon \in[0,\epsilon_1]$$
Let now
$$ \epsilon_2:= \frac{\delta}{2T_\delta(L+\nu(T_\delta,\epsilon_1)^2)}$$
This suffices to claim that
$$\epsilon T_\delta (L+\nu(T_\delta,\epsilon) )^2 \leq \epsilon_2 T_\delta (L+\nu(T_\delta,\epsilon_1) )^2 \leq  \frac{\delta}{2}, \ \ \epsilon \in[0,\epsilon_2]$$
If we now define
$$ \epsilon_\delta:=\min\{\epsilon_1,\epsilon_2\} $$
we obtain the claim. 
\end{proof}

\begin{proposition}\label{prop:Lyap}
Let us consider the dynamical system and assumptions specified in Proposition~\ref{prop:averaging}. Let us further assume that there exits a {\bf twice differentiable} function such that the following assumptions are {\bf  globally and uniformly} satisfied:
\[
	\begin{array}{c}
	c_1\|x\|^2\leq V(z) \leq c_2 \|x\|^2\\
	\frac{\partial V}{\partial x} \bar \phi(x) \leq -c_3 \|x\|^2\\
	\|\frac{\partial V}{\partial z}\|\leq c_4 \|x\|
	\end{array}
\]
Then there exist a function $V'(k,x)$ and $\epsilon_c>0$ such that the following properties are globally and uniformly satisfied for all $\epsilon\in (0,\epsilon_c)$
\[
	\begin{array}{c}
	a_1\|x\|^2\leq V'(k,x) \leq a_2 \|x\|^2\\
	V'(k+1,x+ \epsilon \phi(k,x))-V'(k,x) \leq -\epsilon a_3 \|x\|^2\\
	|V'(k,x)-V'(k,x')| \leq a_4\|x-x'\|  (\|x\|+\|x'\|)
	\end{array}
\]
\end{proposition}
\begin{proof}
We will prove the claim by explicitly costructing $V'(k,x)$ from $V(x)$. Let us define 
$$ V'(k,x):= \sum_{k'=k}^{k+T} V(x(k';k,x)) $$
where $x(k';k,x)$ is the solution of the dynamical system defined above with initial condition $x(k)=x$ at time instant $k$, and $T=T^*$ is a parameter to be determined. We will prove each inequality separately:

\begin{enumerate}
\item Clearly
$$ V'(k,x) = \sum_{k'=k}^{k+T} V(x(k';k,x)) \geq V(x(k;k,x))=V(x)\geq c_1 \|x\|^2 $$
therefore the lower bound is verified for $a_1=c_1$.
Now we observe that
$$\|x(k+1)\|=\|x(k)+ \epsilon \phi(k,x(k))\| \leq (1+\epsilon L)\|x(k)\| \Longrightarrow \|x(k';k,x)\| \leq (1+\epsilon L)^{(k'-k)}\|x\| $$
Also we have
$$ V'(k,x) \leq  \sum_{k'=k}^{k+T} c_2 \|x(k';k,x)\|^2\leq  c_2\left(\sum_{k'=k}^{k+T} (1+\epsilon_c L)^{(k'-k)} \right)\|x\|^2= \underbrace{c_2\left(\sum_{k=0}^{T} (1+\epsilon_c L)^{k} \right)}_{a_2} \|x\|^2 $$
where $a_2$ is independent of $\epsilon$ as long as $\epsilon\in(0,\epsilon_c)$.
\item First note that by definition
\begin{eqnarray*} 
V'(k+1,x+ \epsilon \phi(k,x))-V'(k,x) &=& V(x(k+T+1;k,x)-V(x(k;k,x))\\
&=&V(x(k+T+1;k,x)-V(x) \\
\end{eqnarray*}
and
$$ x(k+T+1;k,x) = x+\epsilon T \overline \phi(x) + g'(k,T,x,\epsilon)$$
Now by using Taylor expansion for $V(x(k+T+1;k,x)$ about the point $x$ and the mean value theorem we obtain
\begin{eqnarray*}  
&&V(x(k+T+1;k,x))-V(x) = V(x)+ \left.\frac{\partial V}{\partial x}\right|_{x}( \epsilon T \overline \phi(x) + g'(k,T,x,\epsilon)) \\
&& + ( \epsilon T \overline \phi(x) + g'(k,T,x,\epsilon))^T\nabla^2 V|_{x'}( \epsilon T \overline \phi(x) + g'(k,T,x,\epsilon)) - V(x)\\
&& \leq -\epsilon T c_3 \|x\|^2 +  c_4 \|g'(k,T,x,\epsilon))\| \|x\| +c_4(\epsilon TL\|x\| +\|g'(k,T,x,\epsilon))\|)^2 \\
&&\leq  -\epsilon T c_3 \|x\|^2 +  c_4 \epsilon T \nu(T,\epsilon) \|x\|^2 + c_4(\epsilon TL +\epsilon T \nu(T,\epsilon))^2\|x\|^2\\
&&\leq -\epsilon T  \left(c_3- c_4(\nu(T,\epsilon)+\epsilon T (L+\nu(T,\epsilon) )^2)\right)\|x\|^2\\
&&= -\epsilon T  \left(c_3- c_4\mu(T,\epsilon)\right)\|x\|^2
\end{eqnarray*}
where $x'=x+\eta(\epsilon T \overline \phi(x) + g'(k,T,x,\epsilon))$ for some $\eta\in[0,1]$. 
If we now let $\delta=\frac{c_3}{2c_4}$ and define $T^*=T_\delta$ and $\epsilon_c=\epsilon_\delta$ as in Proposition~\ref{prop:bound} we have that 
$$ 
V(x(k+T+1;k,x)-V(x) \leq -\epsilon T^* \frac{c_3}{2}\|x\|^2, \ \  \epsilon\in(0,\epsilon_c)$$
therefore the claim is satisfied with $a_3=T^* \frac{c_3}{2}$.

\item The last inequality can be proved by showing that
$$ |V(x(k';k,x))-V(x(k';k,x'))| \leq a'_4\|x-x'\|  (\|x\|+\|x'\|), k\leq k' \leq k+T $$
since from this follows that
$$|V'(k,x)-V'(k,x')| \leq \underbrace{(T+1)a'_4}_{a_4}\|x-x'\|  (\|x\|+\|x'\|)$$

This can be obtained by observing that
$$ V(x(k';k,x))-V(x(k';k,x'))=  \left.\frac{\partial V}{\partial x}\right|_{z}(   x(k';k,x)- x(k';k,x')) $$
where $z=\eta x(k';k,x)+(1-\eta) x(k';k,x')$ for some $\eta\in[0,1]$ by the mean value theorem. From this it follows that by using the properties of the function $V(x)$ that
\begin{eqnarray*}
&&|V(x(k';k,x))-V(x(k';k,x'))| \leq c_4 \| z \| \| x(k';k,x)- x(k';k,x')\| \\
&&\leq c_4  \| x(k';k,x)- x(k';k,x')\|( \| x(k';k,x)\|+ \|x(k';k,x')\|  )
\end{eqnarray*}
The final claim can then be obtained by induction by exploiting the Liptschitz property of the  operator $\phi(k,x)$ necessary to compute the evolutions of $x(k';k,x)$ with respect to the initial conditions similarly to the proof of Theorem~\ref{thm:inv_lyap_finite_conv}.
\end{enumerate}

\end{proof}

\section{Separation of time scales}

In this section we provide the main result.

\begin{proposition}[Semi-global exponential stability]
Consider the system 
\begin{equation}\label{eq:MainSystem}
\left\{
\begin{array}{l}
x(k+1)=x(k)+\epsilon \phi (k, x(k), y(k)), \ \ x(0)=x_0, y(0)=y_0\\
y(k+1)=\varphi (k,y(k),x(k))
\end{array}
\right.
\end{equation}
and assume that is well defined for any $x\in\mathbb{R}^n,y \in \mathbb{R}^n,k\in \mathbb{N}$.
Also assume that the following assumption are satisfied for all $x\in\mathbb{R}^n,y \in \mathbb{R}^n,k\in \mathbb{N}$:
\begin{enumerate} 
\item the map $\varphi (k,y,x)$ satisfy all conditions of Theorem~\ref{thm:inv_layp_exp} and in particular there is function $y^*(x)$ such that
$$y^*(x)= \varphi (k,y^*(x),x), \ \ \forall x,k$$
\item the map $\phi$ is {\bf locally uniformly Liptshitz} and $\phi(k,0,y^*(0))=0$ for all $k$
\item The function
$$\phi_1(k,x) := \phi(k,x,y^*(x)) $$ satisfies all assumptions of Proposition~\ref{prop:averaging} and in particular the function $\phi_1(k,x)$ admits the following average
$$ \bar{\phi}(x) = \lim_{T\to\infty}\frac{1}{T}\sum_{k'=k}^{k+T}\phi_1(k',x), \ \ \forall k$$
\item there exists a twice differentiable function $V(x)$ such that
\[
	\begin{array}{c}
	c_1\|z\|^2\leq V(z) \leq c_2 \|z\|^2\\
	\frac{\partial V}{\partial z} \bar \phi(z) \leq -c_3 \|z\|^2\\
	\|\frac{\partial V}{\partial z}\|\leq c_4 \|z\|
	\end{array}
\]
\end{enumerate}
Then for each $r>0$, there exist $\epsilon_r,C_r,\gamma_r$ possibly function of $r$, such that for all $\|y_0-y^*(0)\|^2+\|x_0\|^2<r^2$ and $\epsilon \in(0,\epsilon_r)$, we have
$$\|x(k)\|^2\leq C_r (1-\epsilon\gamma_r)^k$$ 
\end{proposition}
\begin{proof}
Let $y'(k)$ be defined as
$$
y'(k):=y(k)-y^*(x(k))
$$

We can write that
\begin{align*}
y'(k+1)&=y(k+1)-y^*(x(k+1))\\
&= \varphi (k,y(k),x(k))- y^*(x(k+1))\\
&= \varphi (k,y(k),x(k))- y^*\left(x(k)+\epsilon \phi (k, x(k), y(k))\right)\\
&= \varphi \left(k,y'(k)+y^*(x(k)),x(k) \right)- y^*\left(x(k)+\epsilon \phi (k; x(k), y'(k)+y^*(x(k)))\right)
\end{align*}
Then the dynamics of the original system can be written in this new coordinate system as: 
\begin{equation}
\left\{
\begin{array}{l}
x(k+1)=x(k)+\epsilon \phi (k; x(k), y'(k)+y^*(x(k)))\\
y'(k+1)= \varphi \left(k,y'(k)+y^*(x(k)),x(k) \right)- y^*\left(x(k)+\epsilon \phi (k; x(k), y'(k)+y^*(x(k)))\right)
\end{array}
\right.
\end{equation}
By assumption~(i), then there exists a Lyapunov function $W(k,y,x)$ such that 
\[
	\begin{array}{c}
	b_1\|y'\|^2\leq W(k,y',x) \leq b_2 \|y'\|^2\\
	W(k+1,\varphi(k,y'+y^*(x),x)-y^*(x),x)-W(k,y',x) \leq -b_3 \|y'\|^2\\
	| W(k, y'_1, x) -  W(k, y'_2, x)| \leq b_4 \|y'_1 - y'_2\| (\|y'_1\| + \|y'_2\|)\\
	| W(k, y', x_1) -  W(k, y', x_2)| \leq b_5 \|y'\|^2 \|x_1-x_2\|
	\end{array}
\]
By assumptions~(iii)-(iv), Proposition~\ref{prop:Lyap} guarantees the existence of a Lyapunov function $V'(k,x)$ and constant $\epsilon_c\in(0,1]$ such that 
\[
	\begin{array}{c}
	a_1\|x\|^2\leq V'(k,x) \leq a_2 \|x\|^2\\
	V'(k+1,x+ \epsilon \phi_1(k,x)-V'(k,x) \leq -\epsilon a_3 \|x\|^2, \  \ \ \ \ \epsilon\in(0,\epsilon_c)\\
	|V'(k,x)-V'(k,x')| \leq a_4\|x-x'\|  (\|x\|+\|x'\|)
	\end{array}
\]

Let us define the extended state vector $z=[x^T\ y'^T]^T$ and consider now the global Lyapunov function:
$$ U(k,x,y')= V'(k,x)+W(k,y',x) $$
which has the property 
$$ \min\{a_1,b_1\} (\|x\|^2+\|y'\|^2) \leq U(k,x,y') \leq  \max\{a_2,b_2\} (\|x\|^2+\|y'\|^2) $$
Now, defining $\alpha := \min\{a_1,b_1\}$ and $\beta:= \max\{a_2,b_2\}$ let us consider the sets
$${\color{black}
\mathcal{B}_r:=\left\{(x,y) \, | \, \|x\|^2+\|y'\|^2\leq r \right\}}$$
$$\Omega_r(k):=\left\{(x,y) \, | \, U(k,x,y) \leq r \beta\right\} $$
$${\color{black}\mathcal{B}_{r_0}:=\left\{(x,y) \, | \, \|x\|^2+\|y'\|^2\leq r \frac{\beta}{\alpha}\right\}=\left\{(x,y) \, | \, \|x\|^2+\|y'\|^2\leq r_0\right\},\ r_0 := r \frac{\beta}{ \alpha}}$$

{\color{red}}
which are closed and compact by continuity. Also note that
$$(x_0,y_0-y^*(0))=(x_0,y'_0)\in \mathcal{B}_r \subset \Omega_r(k)\subset \mathcal{B}_{r_0},\ \forall\ k\geq 0$$ {\color{black}where we used the fact that $y'_0=y_0-y^*(0)$.} 
The following basic bounds follow immediately from the Lipschitz and vanishing properties of the various functions when the state is restricted to belong to the compact set $(x,y')\in \mathcal{B}_{r_0}$.
\begin{eqnarray}
\|\phi \left(k, x, y^*(x)\right)\| &=&\|\phi_1(k,x)\| \leq \ell_1 \|x\| \label{eq:bound1a}\\
\|\phi \left(k, x,  y'+y^*(x)\right)- \phi \left(k, x, y^*(x)\right) \| &\leq& \ell_2 \|y'\| \label{eq:bound2a}\\  
\| \varphi(k, y'+y^*(x),x)- y^*(x)\|&=&\| \varphi(k, y'+y^*(x),x)- \varphi(k, y^*(x),x)\|   \nonumber \\
&\leq& \ell_3 \|y' \|  \label{eq:bound4a}\\
\| y^*(x+\epsilon \phi (k, x, y'+y^*(x)))-y^*(x)\| &\leq&  \epsilon\ell_4 \|\phi (k, x, y'+y^*(x))\| \nonumber \\
&\leq&  \epsilon\ell_4 \left(\ell_1\|x\| +  \ell_2\|y'\| \right) \label{eq:bound5a}
\end{eqnarray}
which also imply 
\begin{eqnarray}
\|  \phi \left(k,x, y'+y^*(x)\right)\| &\leq& \|  \phi \left(k,x, y'+y^*(x)\right)-\phi_1(k,x)\|+ \|\phi_1(k,x)\| \nonumber \\
&\leq&  \ell_2 \|y'\| + \ell_1 \|x\| \label{eq:bound3a}   \\
\|x+\epsilon \phi (k, x, y'+y^*(x))\|  &\leq& \|x\|+\epsilon \left(\ell_1\|x\| +  \ell_2\|y'\| \right) \label{eq:bound6a}
\end{eqnarray}

To simplify the notation let us indicate $x=x(k)$, $x^+=x(k+1)$ and
$$\chi=x+\epsilon  \phi \left(k, x,y^*(x) \right)=x+\epsilon\phi_1(k,x) $$

We first want to find an upper bound for the Lyapunov function relative to the slow dynamics:  
\begin{align*}
\Delta V(k,y',x)&:= V'(k+1,x(k+1))-V'(k,x(k))\\
&=V'(k+1,x^+)-V'(k+1,\chi)+V'(k+1,\chi)-V'(k,x)\\
&\leq a_4 \|x^+-\chi\| \left(\|x^+\|+ \|\chi\| \right)-\epsilon a_3 \|x\|^2\\
&\leq  \epsilon a_4\ell_2 \|y'\| (2(1+\epsilon\ell_1)\|x\|+\epsilon\ell_2\|y'\|)-\epsilon a_3 \|x\|^2\\
&\leq \epsilon^2 a_4 \ell_2^2\|y'\|^2+\epsilon 2 a_4\ell_2(1+\epsilon\ell_1)\|x\|\|y'\|-\epsilon a_3 \|x\|^2
\\
&  \leq A_V\|x\|^2+B_V\|x\|\|y\|+C_V\|y'\|^2,
\end{align*}
with $A_V = -\epsilon a_3,\ B_V = \epsilon 2 a_4\ell_2(1+\epsilon\ell_1)$ and $C_V = \epsilon^2 a_4 \ell_2^2$, which, for suitable positive constants $v_{A_1},v_{B_1},v_{B_2},v_{C_1}$, can be rewritten as $A_V = -\epsilon v_{A_1},\ B_V = \epsilon v_{B_1}+\epsilon^2 v_{B_2}$ and $C_V = \epsilon^2 v_{C_1}$.

To simplify the notation let us indicate $y'=y'(k)$, ${y'}^+=y(k+1)$ and
\begin{align*}
\Delta W(k,y',x)& := W(k+1,y'(k+1),x(k+1))-W(k,y'(k),x(k))\\
&= \underbrace{W(k+1,\varphi(k,y'+y^*(x),x)-y^*(x^+),x^+) -W(k+1,\varphi(k,y'+y^*(x),x)-y^*(x),x^+)}_{\Delta W_1}+\\
&+ \underbrace{W(k+1,\varphi(k,y'+y^*(x),x)-y^*(x),x^+) - W(k+1,\varphi(k,y'+y^*(x),x)-y^*(x),x)}_{\Delta W_2}+\\
&+\underbrace{W(k+1,\varphi(k,y'+y^*(x),x)-y^*(x),x)-W(k,y',x)}_{\Delta W_3}
\end{align*}
Using the properties of the Lyapunov function $W$ and Equations \eqref{eq:bound4a} and \eqref{eq:bound5a} 
\begin{align*}
\Delta W_1& \leq b_4 \| y^*(x^+)-y^*(x) \| ( \| \varphi(k,y'+y^*(x),x)-y^*(x^+)\| + \| \varphi(k,y'+y^*(x),x)-y^*(x)\|)\\
& \leq b_4\epsilon\ell_4(\ell_1\|x\|+\ell_2\|y'\|)( 2 \| \varphi(k,y'+y^*(x),x)-y^*(x)\| +  \| y^*(x^+)-y^*(x) \|)\\
&\leq b_4\epsilon\ell_2(\ell_1\|x\|+\ell_2\|y'\|)(2\ell_3\|y'\|+\epsilon\ell_4(\ell_1\|x\|+\ell_2\|y'\|))
\\
&\leq b_4\epsilon\ell_2(\ell_1\|x\|+\ell_2\|y'\|)(\epsilon\ell_1\ell_4\|x\|+(2\ell_3+\epsilon\ell_2\ell_4)\|y'\|)
\\
&\leq\epsilon b_4\ell_2\left[\epsilon\ell_1^2\ell_4\|x\|^2+(2\ell_1\ell_3+2\epsilon\ell_1\ell_2\ell_4)\|y'\|\|x\|+\ell_2(2\ell_3+\epsilon\ell_2\ell_4)\|y'\|^2\right]
\\
&\leq A_{W_1}\|x\|^2+B_{W_1}\|x\|\|y'\|+C_{W_1}\|y'\|^2,
\end{align*}
with $A_{W_1}:= \epsilon^2 b_4\ell_1^2\ell_2\ell_4,\ B_{W_1} := \epsilon b_4\ell_2(2\ell_1\ell_3+2\epsilon\ell_1\ell_2\ell_4)$ and $C_{W_1} := \epsilon b_4\ell_2^2(2\ell_3+\epsilon\ell_2\ell_4)$.

Concerning $\Delta W_2$, using again the properties of $W$ and Equations \eqref{eq:bound4a} and \eqref{eq:bound6a}
\begin{align*}
\Delta W_2& \leq b_5 \| \varphi(k,y'+y^*(x),x)-y^*(x)\|^2 \|x^+-x\|\\
&\leq b_5\ell_3^2\|y'\|^2\epsilon(\ell_1\|x\|+\ell_2\|y'\|)
\end{align*}
Now, considering that $\|x\|$ and $\|y\|$ are smaller than a $\bar{r}$ and that $y^*$ is a twice differentiable function of $x$, we have that $\|y^*\|$ is bounded itself by a $\tilde{r}$, and so
\begin{align*}
\Delta W_2
&\leq b_5\ell_3^2\|y'\|\|y-y^*(x)\|\|\epsilon(\ell_1\|x\|+\ell_2\|y'\|)
\\
&\leq b_5\ell_3^2\|y'\|(\bar{r}+\tilde{r})\epsilon(\ell_1\|x\|+\ell_2\|y'\|)
\\
&\leq b_5\ell_3^2(\bar{r}+\tilde{r})\epsilon\|y'\|(\ell_1\|x\|+\ell_2\|y'\|)
\\
&\leq \epsilon b_5\ell_3^2(\bar{r}+\tilde{r})(\ell_1\|x\|\|y'\|+\ell_2\|y'\|^2)
\\
&\leq B_{W_2}\|x\|\|y'\|+C_{W_2}\|y'\|^2,
\end{align*}
with $B_{W_2}:=\epsilon b_5\ell_1\ell_3^2(\bar{r}+\tilde{r})$ and $C_{W_2}:=\epsilon b_5\ell_2\ell_3^2(\bar{r}+\tilde{r})$.

Finally, using the properties of $W$, we have
$$
\Delta W_3\leq-b_3\|y'\|\leq -C_{W_3}\|y'\|^2,
$$
with $C_{W_3}:=b_3$.
Summing up all the previous inequalities, we have
\[
\Delta W(k,y',x)\leq A_W\|x\|^2+B_W\|x\|\|y'\|+C_W\|y'\|^2,
\]
with
\begin{align*}
A_W &= A_{W_1} =  \epsilon^2 b_4\ell_1^2\ell_2\ell_4 = \epsilon^2 w_{A_1}
\\
B_W &= B_{W_1}+B_{W_2} = \epsilon b_4\ell_2(2\ell_1\ell_3+2\epsilon\ell_1\ell_2\ell_4)+\epsilon b_5\ell_1\ell_3^2(\bar{r}+\tilde{r}) = \epsilon w_{B_1}+\epsilon^2 w_{B_2}
\\
C_W &= C_{W_1}+C_{W_2}-C_{W_3} =\epsilon b_4\ell_2^2(2\ell_3+\epsilon\ell_2\ell_4)+\epsilon b_5\ell_2\ell_3^2(\bar{r}+\tilde{r})-b_3 = -w_{C_1}+\epsilon w_{C_2}+\epsilon^2 w_{C_3},
\end{align*}
for suitable positive constants $w_{A_1},w_{B_1},w_{B_2},w_{C_1},w_{C_2},w_{C_3}$.

Now it is possible to evaluate $\Delta U(k,x,y')= U(k+1,x^+,y'^+)-U(k,x,y')$. It holds
\[
\Delta U(k,x,y')= \Delta W(k,y',x) +\Delta V(k,y',x),
\]
and the following bound easily follows
\begin{align*}
\Delta U(k,x,y')&\leq  A_V\|x\|^2+B_V\|x\|\|y\|+C_V\|y'\|^2+ A_W\|x\|^2+B_W\|x\|\|y'\|+C_W\|y'\|^2
\\
&\leq (A_V+A_W)\|x\|^2+(B_V+B_W)\|x\|\|y\|+(C_V+C_W)\|y\|^2.
\end{align*}
This quadratic form can be rewritten as
\[
\Delta U(k,x,y')\leq \begin{bmatrix}
\|x\| & \|y'\| 
\end{bmatrix} \underbrace{\begin{bmatrix}
A_V+A_W & \frac{1}{2}(B_V+B_W) \\
\frac{1}{2}(B_V+B_W) & C_V+C_W
\end{bmatrix}}_{Q_U}\begin{bmatrix}
\|x\| \\ \|y'\| 
\end{bmatrix}.
\]
Now the aim is to show that $\Delta U(k,x,y')$ is always a negative quantity for a small enough choice of the parameter $\epsilon$. It is therefore necessary to study whether matrix $Q_U$ is negative definite:
\begin{align*}
Q_U &= \begin{bmatrix}
-\epsilon v_{A_1}+\epsilon^2 w_{A_1} & \frac{1}{2}(\epsilon v_{B_1}+\epsilon^2 v_{B_2}+\epsilon w_{B_1}+\epsilon^2 w_{B_2}) \\
\frac{1}{2}(\epsilon v_{B_1}+\epsilon^2 v_{B_2}+\epsilon w_{B_1}+\epsilon^2 w_{B_2}) & \epsilon^2 v_{C_1}-w_{C_1}+\epsilon w_{C_2}+\epsilon^2 w_{C_3}
\end{bmatrix}
\end{align*}
In order to verify whether $Q_U$ is negative definite, it is enough to verify whether the first principal minor is negative and the second one is positive for some choices of $\epsilon$. These quantities, in Lindau notations, are respectively
\[
	-\epsilon v_{A_1}+o(\epsilon),\quad \epsilon v_{A_1}w_{C_1}+o(\epsilon)
\]
and so, there exists $\epsilon'_r$ such that for $\epsilon<\epsilon'_r$ matrix $Q_U$ is negative definite. As a consequence, there exists a positive constant $\ell_U$ such that
\[
	Q_U\leq -\epsilon\ell_U I,
\]
and so 
\begin{align*}
\Delta U(k,x,y')&\leq  -\epsilon\ell_U (\|x\|^2+\|y'\|^2)
\\
&\leq -\epsilon\ell_U \left(\frac{1}{c_1} V(x)+\frac{1}{b_1}W(k,y',x)\right)
\\
&\leq -\epsilon \gamma_r U(k,x,y'),
\end{align*}
with $\gamma_r:=\ell_U \max\{\frac{1}{c_1},\frac{1}{b_1}\}$.

From the latter inequality it follows, for some $\ell>0$, 
\[
\|x(k)\|^2 + \|y'(k)\|^2
 \leq \ell (1-\epsilon\gamma_r)^{k}
\left(\|x(0)\|^2 + \|y'(0)\|^2\right),
\]
and since $\|x(0)\|^2+\|y'(0)\|\leq r$, there exists a $C_r>0$ such that
\[
\|x(k)\|^2 \leq C_r (1-\epsilon\gamma_r)^k, \ \ \ \ \epsilon\in(0,\epsilon_r)
\]
where $\epsilon_r=\min\{\epsilon_c,\epsilon_r'\}$.
\end{proof}

\begin{proposition}[Global exponential stability]
Consider the system defined in the previous theorem and assume that all assumptions hold. If in addition the following inequalities are globally and uniformly satisfied:
\begin{eqnarray*}
\|\phi \left(k, x, y^*(x)\right)\| &=&\ell_1 \|x\| \\
\|\phi \left(k, x,  y'+y^*(x)\right)- \phi \left(k, x, y^*(x)\right) \| &\leq& \ell_2 \|y'\| \\  
\| \varphi(k, y'+y^*(x),x)- y^*(x)\|&=&  \ell_3 \|y' \| \\
\| y^*(x+\epsilon \phi (k, x, y'+y^*(x)))-y^*(x)\| &\leq&  \epsilon\ell_4  \|\phi (k, x, y'+y^*(x))\|
\end{eqnarray*}
Then there exist $\epsilon_c$ such that for $\epsilon\in(0,\epsilon_c)$ the system is globally exponentially stable. 
\end{proposition}

\end{document}